\documentclass[10pt,francais,leqno]{smfart}

\usepackage[T1]{fontenc}
\usepackage[francais]{babel}
\usepackage{enumitem}
\usepackage{bbm}

\usepackage{amssymb,url,xspace,smfthm}
\usepackage{mathrsfs,euscript,color}

\usepackage[all]{xy}

\newcommand{\BibTeX}{{\scshape Bib}\kern-.08em\TeX}
\newcommand{\T}{\S\kern .15em\relax }
\newcommand{\AMS}{$\mathcal{\rm a}$\kern-.1667em\lower.5ex\hbox
        {$\mathcal{M}$}\kern-.125em$\mathcal{S}$}

\theoremstyle{plain}

\newtheorem*{mapropo}{\textsc{Proposition}}
\newtheorem*{mapropo1}{\textsc{Proposition 1}}
\newtheorem*{mapropo2}{\textsc{Proposition 2}}

\newtheorem*{monlem}{\textsc{Lemme}}

\newtheorem*{marema}{\textsc{Remarque}}

\catcode`\@=11
\catcode`\Ž=13
\catcode`\=13
\catcode`\ˆ=13
\catcode`\=13
\catcode`\=13
\catcode`\‰=13
\catcode`\™=13
\catcode`\"=13
\catcode`\ž=13
\catcode`\=13
\catcode`\'=13
\def Ž{\'e}
\def {\`e}
\def ˆ{\`a}
\def {\`u}
\def {\^e}
\def ‰{\^a}
\def ™{\^o}
\def "{\^{\i}}
\def ž{\^u}
\def {\c c}
\def '{\"e}
\catcode`\:=13
\def :{~\string:}
\catcode`\;=13
\def ;{~\string;}
\catcode`\!=13
\def !{~\string!}

\def\ES#1{\EuScript{#1}}

\def\bs#1{\boldsymbol{#1}}
\def\cad{c'est-\`a-dire\ }

\title[DONN\'EES ENDOSCOPIQUES: APPLICATIONS D'UNE CONSTRUCTION DE LANGLANDS]
{DONN\'EES ENDOSCOPIQUES \\ D'UN GROUPE R\'EDUCTIF CONNEXE: \\ APPLICATIONS D'UNE CONSTRUCTION DE LANGLANDS}

\author{Bertrand Lemaire}
\address{Institut de Math\'ematiques de Luminy et
UMR 6206 du CNRS, 
Universit\'e Aix--Marseille II, Case Postale 907\\
F--13288 Marseille Cedex 9}
\email{Bertrand.Lemaire@univ-amu.fr}

\author{Jean-Loup Waldspurger}
\address{CNRS, Institut de Mathématiques de Jussieu-Paris-Rive-Gauche, 
2 place Jussieu 75005 Paris}
\email{jean-loup.waldspurger@imj-prg.fr}

\begin{document}



\altkeywords{non-twisted ebdoscopy, elliptic endoscopic datum, localization}
\maketitle
\setcounter{tocdepth}{3}

\begin{abstract}
Soient $F$ un corps global, et $G$ un groupe rŽductif connexe dŽfini sur $F$. On prouve que si deux donnŽes endoscopiques de $G$ sont Žquivalentes 
en presque toute place de $F$, alors elles sont Žquivalentes. Le rŽsultat est encore vrai pour l'endoscopie (ordinaire) avec caractre. On donne aussi, pour $F$ global ou local 
et $G$ quasi-simple simplement connexe, une description des donnŽes endoscopiques elliptiques de $G$.  
\end{abstract}

\begin{altabstract}Let $F$ be a global field, and $G$ a connected reductive group defined over $F$. We prove that two endoscopic data of $G$ which are equivalent almost everywhere, 
are equivalent. The result remains true for (non-twisted) endoscopy with character. We also give, for $F$ global or local and $G$ quasi-simple simply connected, 
a description of the elliptic endoscopic data of $G$. 
\end{altabstract}

\section*{Introduction}Soit $F$ un corps global, et soit $G$ un groupe rŽductif connexe dŽfini sur $F$. On fixe une 
cl™ture sŽparable algŽbrique $\overline{F}$ de $F$. On note $\Gamma_F$ le groupe de Galois de $\overline{F}/F$, et $W_F$ son groupe de Weil. 
Pour $v$ une place de $F$, on note $F_v$ le complŽtŽ de $F$ en $v$, et on pose $G_v= G\times_F F_v$. Fixons un ŽlŽment 
$\bs{a}\in {\rm H}^1(W_F,Z(\hat{G}))/\ker^1(W_F,Z(\hat{G}))$ --- cf. \ref{le rŽsultat}. 
Soit $\bs{G}'=(G'\!,\ES{G}'\!,s)$ une donnŽe endoscopique de $G$. Pour chaque place $v$ de $F$, cette donnŽe dŽfinit par localisation une donnŽe endoscopique 
 $\bs{G}_v=(G'_v,\ES{G}'_v,s)$ de $(G_v,\bs{a}_v)$. Le premier rŽsultat prouvŽ ici est:
 
 \begin{mapropo1}
Si $\bs{G}'_1$ et $\bs{G}'_2$ sont deux donnŽes endoscopiques 
 de $(G,\bs{a})$ telles que les donnŽes locales $\bs{G}'_{1,v}$ et $\bs{G}'_{2,v}$ soient Žquivalentes pour presque tout $v$, alors les donnŽes $\bs{G}'_1$ et 
 $\bs{G}'_2$ sont Žquivalentes.
 \end{mapropo1}
 L'ŽlŽment $\bs{a}$ ne joue en fait aucun r™le puisqu'on se ramne facilement au cas o $G$ est $F$-quasi-simple et simplement connexe et 
 (forcŽment) $\bs{a}=1$. Par restriction des scalaires ˆ la Weil, on se ramne ensuite au cas o $G$ est (absolument) quasi-simple et simplement connexe. 
 On peut donc supposer, et l'on suppose, que le groupe dual $\hat{G}$ est adjoint et simple. Comme les constructions que nous allons faire sont valables aussi bien 
 dans le cas global que dans le cas local, la lettre $F$ dŽsigne maintenant un corps commutatif qui est soit un corps global soit un corps local. Enfin pour allŽger l'Žcriture, 
 on considre plut™t la forme \og galoisienne\fg des $L$-groupes. L'adaptation des constructions ci-dessous 
 ˆ leur forme \og groupes de Weil\fg ne pose aucun problme. 
 
 Le groupe dual $\hat{G}=\hat{G}_{\rm AD}$ est muni d'une action galoisienne $\sigma \mapsto \sigma_G$. On fixe une paire de Borel 
 $(\hat{B},\hat{T})$ de $\hat{G}$ conservŽe par cette action galoisienne. On note $\Delta$ l'ensemble de racines simples dŽterminŽ par $\hat{B}$, et $\alpha_0$ 
 l'opposŽe de la plus grande racine positive. On pose $\Delta_{\rm a}= \{\alpha_0\}\cup \Delta$. On note $\ES{D}$, resp. $\ES{D}_{\rm a}$, le diagramme de Dynkin dont l'ensemble 
 des sommets est $\Delta$, resp. $\Delta_{\rm a}$. L'ensemble $\Omega$ 
 des ŽlŽments de $W=N_{\hat{G}}(\hat{T})/ \hat{T}$ qui conservent $\Delta_{\rm a}$ s'identifie ˆ un sous-groupe abŽlien distinguŽ du groupe d'automorphismes ${\rm Aut}(\ES{D}_{\rm a})$ de 
 $\ES{D}_{\rm a}$, et on a la dŽcomposition
 $$
 {\rm Aut}(\ES{D}_{\rm a})= \Omega \rtimes {\rm Aut}(\ES{D}).
 $$
 Puisque l'action galoisienne conserve la paire $(\hat{B},\hat{T})$, elle induit une action sur $\ES{D}$. 
 On note $E$ l'extension galoisienne finie de $F$ telle que $\Gamma_E = \ker (\Gamma_F \rightarrow {\rm Aut}(\ES{D}))$. 
 
 Soit $\bs{G}'= (G'\!,\ES{G}'\!,s)$ une donnŽe endoscopique de $G$ (on a donc $\ES{G}'\subset {^LG}= \hat{G} \rtimes \Gamma_F$). On suppose que 
 $s$ est d'ordre fini $d$. 
 \`A Žquivalence prs, on peut supposer $s\in \hat{T}$. On fixe un sous-groupe 
 de Borel $\hat{B}'$ de $\hat{G}'= Z_{\hat{G}}(s)^\circ$. Pour tout $\sigma\in \Gamma_F$, on peut choisir un ŽlŽment 
 $(g(\sigma),\sigma)\in \ES{G}'$ dont l'action par conjugaison conserve la paire de Borel $(\hat{B}'\!,\hat{T})$ de $\hat{G}'$. 
 Il dŽtermine un ŽlŽment $w_{G'}(\sigma)\in W$, ainsi qu'une action galoisienne sur $\hat{T}$
$$
\sigma \mapsto \sigma_{G'}= w_{G'}(\sigma)\sigma_G.
$$
Son image est contenue dans le groupe d'automorphismes de $\hat{T}$ qui fixent $s$ et conservent $\hat{B}'$. Le point-clŽ de la construction est une cons\'equence des 
rŽsultats de Langlands \cite{L}. \`A $s$ et $\hat{B}'$, Langlands associe un ensemble de racines $\mathfrak{X}= \bigcup_{k=0,\ldots ,d-1}\mathfrak{X}_k$. Il d\'emontre qu'\`a conjugaison prs par un 
ŽlŽment de $W$,  cet ensemble est soit $\Delta$ soit $\Delta_{\rm a}$. Quitte ˆ conjuguer $\hat{G}'$ par un ŽlŽment $u\in N_{\hat{G}}(\hat{T})$ et ˆ remplacer $\hat{B}'$ par $u(\hat{B}')$, 
on peut supposer que $\mathfrak{X}= \Delta$ ou $\mathfrak{X}= \Delta_{\rm a}$. Alors l'application $\sigma \mapsto w_{G'}(\sigma)$ 
est ˆ valeurs dans $\Omega$, et $w_{G'}(\sigma)$ conserve $\mathfrak{X}$. Notons $K$ l'extension galoisienne finie de $F$ telle que 
$\Gamma_K$ soit le noyau de l'homomorphisme $\sigma \mapsto \sigma_{G'}$. On distingue deux cas:
\begin{itemize}
\item {\bf Cas 1}: si $\mathfrak{X}= \Delta$, alors 
$w_{G'}(\sigma)= 1$ et $\ES{G}'= \hat{G}'\rtimes \Gamma_F$.
\item {\bf Cas 2}: si $\mathfrak{X}= \Delta_{\rm a}$, alors $E\subset K$ et la restriction de l'application $\sigma\mapsto w_{G'}(\sigma)$ \`a $\Gamma_{E}$ 
se quotiente en  une injection de $\Gamma_{K/E}= {\rm Gal}(K/E)$ dans $\Omega$.
\end{itemize}
ConsidŽrons deux donnŽes endoscopiques $\bs{G}'_1=(G'_1,\ES{G}'_1,s)$ et $\bs{G}'_2= (G'_2,\ES{G}'_2,s)$ de $G$, avec le mme $s$ d'ordre fini $d$. On note encore 
$\hat{G}'$ le groupe $Z_{\hat{G}}(s)^\circ$. On applique les constructions ci-dessus ˆ ces deux donnŽes, en choisissant le mme sous-groupe de Borel $\hat{B}'$ de $\hat{G}'$ 
pour les deux constructions. On obtient les mmes ensembles $\mathfrak{X}$ et $\mathfrak{X}_k$ pour les deux donnŽes, seules diffrent (Žventuellement) 
les actions galoisiennes $\sigma_{G'_1}$ et $\sigma_{G'_2}$. On obtient:
\begin{itemize}
\item {\bf Cas 1}: si $\mathfrak{X}=\Delta$, les deux donn\'ees sont \'equivalentes.
\item {\bf Cas 2}: si $\mathfrak{X}=\Delta_{\rm a}$, les donn\'ees $\bs{G}'_{1}$ et $\bs{G}'_{2}$ sont \'equivalentes si et seulement s'il existe $\omega\in \Omega$ tel que 
$\omega(\mathfrak{X}_{k})=\mathfrak{X}_{k}$ pour $k= 1,\ldots,d-1$ et $\omega \sigma_{G'_{1}}\omega^{-1}=\sigma_{G'_{2}}$ pour tout $\sigma\in \Gamma_{F}$.
\end{itemize}
Gr‰ce ˆ ce rŽsultat, on dŽmontre la proposition 1 dans le cas o les ŽlŽments $s_1$ et $s_2$ sont d'ordre fini. 
Notons qu'il suffit de traiter le cas $s_1=s_2=s$. On supprime ensuite la restriction sur $s$ gr‰ce aux constructions de Langlands \cite{L} pp. 704-705: 
elles permettent de construire des donnŽes endoscopiques $\bs{G}''_1= (G''_1,\ES{G}''_1,t)$ et $\bs{G}''_2= (G''_2,\ES{G}''_2,t)$ de $G$ 
avec $t$ d'ordre fini, telles que les donnŽes $\bs{G}'_1$ et $\bs{G}'_2$ sont Žquivalentes si et seulement si les donnŽes $\bs{G}''_1$ et $\bs{G}''_2$ 
le sont (idem pour les donnŽes locales $\bs{G}'_{i,v}$ et $\bs{G}''_{i,v}$).  

\vskip1mm
Les constructions ci-dessus permettent aussi de dŽcrire les donnŽes endoscopiques elliptiques de $G$. Les hypothses sont les mmes qu'avant: $F$ est global ou local; 
$\hat{G}$ est adjoint et simple; et on considre la forme \og galoisienne \fg des $L$-groupes. Notons $\underline{\ES{E}}(G)$ l'ensemble des couples $( \omega_{G'}, \ES{O})$ o\`u 
$\omega_{G'}:\Gamma_{F}\to \Omega$ est une application telle que l'application
$$
\sigma\mapsto \sigma_{G'}\buildrel {\rm d\acute{e}f}\over{=}\omega_{G'}(\sigma)\sigma_{G}
$$ soit un homomorphisme de $\Gamma_{F}$ dans ${\rm Aut}(\ES{D}_{\rm a})$, et 
$\ES{O}$ est un sous-ensemble non vide de $\Delta_{\rm a}$ qui est conserv\'e par l'action $\sigma\mapsto \sigma_{G'}$ et qui forme une unique orbite pour cette action.
Deux tels couples $(\omega_{G'_{1}},\ES{O}_{1})$ et $(\omega_{G'_{2}},\ES{O}_{2})$ sont dits \'equivalents s'il existe $\omega\in \Omega$ tel que $\ES{O}_{2}=\omega(\ES{O}_{1})$ et $\sigma_{G'_{2}}=\omega\sigma_{G'_{1}}\omega^{-1}$ pour tout $\sigma\in \Gamma_{F}$. Notons $\underline{E}(G)$ l'ensemble des classes d'\'equivalence dans $\underline{\ES{E}}(G)$. 

ConsidŽrons un couple $(\omega_{G'},\ES{O})\in \underline{\ES{E}}(G)$. Posons
$$
d=\sum_{\alpha\in \ES{O}}d(\alpha).
$$
Fixons $\zeta_d\in {\Bbb C}^\times$ une racine primitive de $1$ d'ordre $d$. 
Soit $s$ l'unique \'el\'ement de $\hat{T}$ tel que $\alpha(s)=1$ pour tout $\alpha\in \Delta\smallsetminus (\ES{O}\cap \Delta)$ et $\alpha(s)=\zeta_{d}$ pour tout 
$\alpha\in \ES{O}\cap \Delta$. Il est fixe par l'action galoisienne $\sigma_{G'}$. 
Posons $\hat{G'}=Z_{\hat{G}}(s)^\circ$. Notons $\ES{G}'$ le sous-ensemble de $^LG$ form\'e des \'el\'ements $(g\tilde{\omega}_{G'}(\sigma),\sigma)$ pour $g\in \hat{G'}$ et $\sigma\in \Gamma_{F}$, o\`u $\tilde{\omega}_{G'}(\sigma)$ est un repr\'esentant quelconque de $\omega_{G'}(\sigma)$ dans $N_{\hat{G}}(\hat{T})$. L'ensemble $\ES{G}'$ est un  groupe qui normalise $\hat{G}'$. On en d\'eduit de la fa\c{c}on habituelle une $L$-action de $\Gamma_{F}$ sur $\hat{G}'$. Soit $G'$ un groupe rŽductif connexe dŽfini et quasi-dŽployŽ sur $F$ 
dont $\hat{G}'$, muni de cette action galoisienne, est le groupe dual. Le triplet $\bs{G}'=(G'\!,\ES{G}'\!,s)$ est une donn\'ee endoscopique de $G$. 
L'hypothse que $\ES{O}$ est une seule orbite sous l'action de $\Gamma_F$ implique que cette donnŽe est elliptique. Le second rŽsultat prouvŽ ici est:
\begin{mapropo2}
L'application qui, \`a un couple $(\omega_{G'},\ES{O})\in \underline{\ES{E}}(G)$, associe la donn\'ee $\bs{G}'\!$, se quotiente en une bijection de $\underline{E}(G)$ sur l'ensemble des classes d'\'equivalence de donn\'ees endoscopiques elliptiques de $G$.
\end{mapropo2}

L'article comporte deux sections, la premire est dŽdiŽe ˆ la preuve de la proposition 1, la seconde ˆ celle de la proposition 2. 
\`A partir de \ref{hypothses et dŽfinitions}, le groupe $G$ est (absolument) quasi-simple et simplement connexe, et on utilise la forme \og galoisienne\fg des $L$-groupes.

\section{\'Equivalence presque partout de donnŽes endoscopiques}\label{Žquivalence pp}

\subsection{Le rŽsultat}\label{le rŽsultat}
On consid\`ere un corps commutatif $F$ qui est soit un corps global (corps de nombres ou corps de fonctions) soit un corps local (${\Bbb R}$, ${\Bbb C}$, une extension finie de ${\Bbb Q}_p$ ou 
de ${\Bbb F}_p((t))$). On fixe une cl\^oture sŽparable 
algŽbrique $\overline{F}$ de $F$. On note $\Gamma_{F}$ le groupe de Galois de $\overline{F}/F$, et $W_F$ son groupe de Weil. Via l'homomorphisme naturel 
$W_F\rightarrow \Gamma_F$, le groupe $W_F$ agit sur tout ensemble sur lequel agit $\Gamma_F$. Toutes les extensions sŽparables finies de $F$ seront suppos\'ees incluses dans $\overline{F}$. Pour une extension sŽparable finie $K$ de $F$, on note $\Gamma_K$ le groupe de Galois de $\overline{F}/K$, et $W_K$ son groupe de Weil. Si de plus $K/F$ est galoisienne, on note $\Gamma_{K/F}\;(= \Gamma_F/\Gamma_K)$ le groupe de Galois de $K/F$.

Soit $G$ un groupe r\'eductif connexe d\'efini sur $F$. On note $\hat{G}$ le groupe dual de $G$. Il est muni d'une action galoisienne notŽe $\sigma\mapsto \sigma_{G}$. On note 
${^LG}= \hat{G} \rtimes W_F$ le $L$-groupe de $G$, et $Z(\hat{G})$ le centre de $\hat{G}$. Dans le cas local on fixe un ŽlŽment 
$\bs{a}\in {\rm H}^1(W_F,Z(\hat{G}))$, et dans le cas global on fixe $\bs{a}\in {\rm H}^1(W_F,Z(\hat{G}))/\ker^1(W_F,Z(\hat{G}))$, o 
$\ker^1(W_F,Z(\hat{G}))$ est le noyau de l'homomorphisme de localisation
$$
{\rm H}^1(W_F,Z(\hat{G}))\rightarrow \prod_v {\rm H}^1(W_{F_v},Z(\hat{G})).
$$
D'aprs un thŽorme de Langlands, ˆ cet ŽlŽment $\bs{a}$ correspond un caractre (i.e. un homomorphisme continu dans ${\Bbb C}^\times$) $\bs{\omega}$ de $G(F)$ --- cf. \cite{LL,W}. On peut supposer, mais ce n'est pas nŽcessaire, que $\bs{\omega}$ est trivial sur $Z(G;F)$; sinon la thŽorie est vide. Ici $Z(G;F)$ est le groupe des points $F$--rationnels du centre de $G$. 
 Nous utilisons la notion de {\it donnŽe endoscopique de $(G,\bs{a})$} telle qu'elle est dŽfinie dans \cite[I.1.5, VI.3.1]{MW}\footnote{Le peu que nous utilisons de loc.~cit. est valable en toute caractŽristique.}. Il s'agit donc d'endoscopie ordinaire (i.e. non tordue) avec caractre. Une telle donnŽe est notŽe $\bs{G}'=(G'\!,\ES{G}'\!,s)$. On dŽfinit la notion d'Žquivalence, ou d'isomorphisme, entre deux telles donnŽes, et en particulier, le groupe ${\rm Aut}(\bs{G}')$ des automorphismes de $\bs{G}'$. Ce groupe contient $\hat{G}'=\ES{G}'\cap \hat{G}$. Rappelons que $\hat{G}'$ est la composante neutre $Z_{\hat{G}}(s)^\circ$ du centralisateur de $s$ dans $\hat{G}$ (munie d'une action galoisienne $\sigma \mapsto \sigma_{G'}$ qui en 
fait un groupe dual de $G'$).

On suppose maintenant que $F$ est un corps global. Pour chaque place $v$ de $F$, on fixe un prolongement $\overline{v}$ de $v$ \`a $\overline{F}$. 
Le fixateur $\Gamma_v$ de $\overline{v}$ dans $\Gamma_{F}$ s'identifie \`a $\Gamma_{F_{v}}$ o\`u $F_{v}$ est le compl\'et\'e de $F$ en $v$. On a aussi une 
identification $W_{F_v}\subset W_F$. Soit $\bs{G}'=(G'\!,\ES{G}'\!,s)$ une donn\'ee endoscopique de $(G,\bs{a})$. Pour chaque place $v$ de $F$, 
on en d\'eduit par localisation une donn\'ee endoscopique $\bs{G}'_v=(G'_{v},\ES{G}'_{v},s)$ de $(G_{v},\bs{a}_v)$, 
o $G_v=G\times_F F_v$ et $\bs{a}_v \in {\rm H}^1(W_{F_v},Z(\hat{G}))$. 
En particulier, le caractre $\bs{\omega}_v$ de $G(F_v)$ correspondant ˆ $\bs{a}_v$ est le prolongement (par continuitŽ) de $\bs{\omega}$, et $\ES{G}'_{v}$ est le sous-groupe des $(g,w)\in \ES{G}'$ tels que $w\in W_{F_v}$.

La proposition suivante est le rŽsultat principal de cette section \ref{Žquivalence pp}.

\begin{mapropo}
Consid\'erons deux donn\'ees endoscopiques $\bs{G}'_{1}=(G'_{1},\ES{G}'_{1},s_1)$ et $\bs{G}'_{2}=(G'_{2},\ES{G}'_{2},s_2)$ de $(G,\bs{a})$. On suppose que 
pour presque toute place $v$ de $F$, les donnŽes $\bs{G}'_{1,v}$ et $\bs{G}'_{2,v}$ sont \'equivalentes.
Alors les donnŽes $\bs{G}'_{1}$ et $\bs{G}'_{2}$ sont \'equivalentes. 
\end{mapropo}

\begin{marema}
{\rm 
On a besoin en fait que les donnŽes locales soient Žquivalentes en un ensemble de places auquel s'applique le thŽorme de Tchebotarev, 
c'est-ˆ-dire qui soit analytiquement dense. Sinon, il est possible de produire un groupe $G$ et deux donnŽes endoscopiques  elliptiques  
$\bs{G}'_1$ et $\bs{G}'_2$ de $G$ (avec $\bs{a}=1$) qui vŽrifient les conditions suivantes:
\begin{itemize}
\item les donnŽes locales $\bs{G}'_{1,v}$ et $\bs{G}'_{2,v}$ sont simultanŽment elliptiques ou non elliptiques, 
et l'ensemble des places $v$ de $F$ o elles sont elliptiques est infini;
\item les donnŽes locales $\bs{G}'_{1,v}$ et $\bs{G}'_{2,v}$ sont Žquivalentes en toute place $v$ de $F$ o elles sont 
elliptiques;
\item les donnŽes $\bs{G}_1$ et $\bs{G}_2$ sont non Žquivalentes.
\end{itemize}
}
\end{marema}

\subsection{RŽduction au cas o $\hat{G}=\hat{G}_{\rm AD}$ est simple}\label{rŽduction}Revenons ˆ $F$ global ou local. Soit $G_{\rm SC}$ le revtement simplement connexe du groupe dŽrivŽ de $G$. Son groupe dual est le groupe adjoint $\hat{G}_{\rm AD}= \hat{G}/Z(\hat{G})$, et toute donnŽe endoscopique $\bs{G}'=(G'\!,\ES{G}'\!,s)$ de $(G,\bs{a})$ dŽfinit 
une donnŽe endoscopique $\smash{\overline{\bs{G}}}'=(\smash{\overline{G}},\smash{\overline{\ES{G}}}',\bar{s})$ de $G_{\rm SC}$. L'ŽlŽment $\bar{s}$ est l'image de $s$ dans $\hat{G}_{\rm AD}$, 
la composante neutre de son centralisateur $Z_{\hat{G}_{\rm AD}}(\bar{s})^\circ= \hat{G}'/Z(\hat{G})$ est munie de l'action galoisienne dŽduite de $\sigma_{G'}$, $\smash{\overline{G}}'$ est un $F$-groupe rŽductif connexe quasi-dŽployŽ dont $\hat{G}'/Z(\hat{G})$ est le groupe dual, et $\smash{\overline{\ES{G}}}'= \ES{G}'/Z(\hat{G})$.

Soient $\bs{G}'_{1}=(G'_{1},\ES{G}'_{1},s_1)$ et $\bs{G}'_{2}=(G'_{2},\ES{G}'_{2},s_2)$ deux donnŽes endoscopiques de $(G,\bs{a})$. 
Elles dŽfinissent comme ci-dessus des donnŽes endoscopiques $\smash{\overline{\bs{G}}}'_1$ et $\smash{\overline{\bs{G}}}'_2$ 
de $G_{\rm SC}$. Par dŽfinition, un ŽlŽment $x$ de $\hat{G}$ est une Žquivalence entre 
$\bs{G}'_1$ et $\bs{G}'_2$ si et seulement si l'image $\bar{x}$ de $x$ dans $\hat{G}_{\rm AD}$ est une Žquivalence entre $\smash{\overline{\bs{G}}}'_1$ et $\smash{\overline{\bs{G}}}'_2$. 
Si $F$ est global, l'application $\bs{G}'\mapsto \smash{\overline{\bs{G}}}'$ commute ˆ la localisation. 
On en dŽduit qu'il suffit de prouver la proposition de \ref{le rŽsultat} dans le cas o $G$ est semisimple simplement connexe et (forcŽment) 
$\bs{a}=1$. 

On suppose dŽsormais $G=G_{\rm SC}$ (et toujours $F$ global ou local). On fixe une paire de Borel $(\hat{B},\hat{T})$ de $\hat{G}$ conservŽe par l'action galoisienne $\sigma \mapsto \sigma_G$. 
On note $\Sigma$ l'ensemble des racines de $\hat{T}$ dans $\hat{G}$ et $\Sigma^+$, resp. $\Delta$, le sous-ensemble de racines positives, resp. simples, d\'etermin\'e par $\hat{B}$.  Notons $\ES{D}$ le diagramme de Dynkin dont l'ensemble de sommets est $\Delta$. Puisque la paire $(\hat{B},\hat{T})$ est stable par l'action galoisienne, cette dernire induit une 
action sur $\ES{D}$ et peut-tre considŽrŽe comme un homomorphisme $\Gamma_F \rightarrow {\rm Aut}(\ES{D})$. On a la dŽcomposition
$$
\hat{G}= \hat{G}_1\times \cdots \times\hat{G}_d
$$
o $\hat{G}_i$ est simple et adjoint. On note 
$$
G= G_1\times \cdots \times G_d
$$
la dŽcomposition duale: $G_i$ est un groupe semisimple simplement connexe, dŽfini sur $\overline{F}$ et (absolument) quasi-simple. Pour 
$i=1,\ldots ,d$, on note $\ES{D}_i$ la composante connexe de $\ES{D}$ correspondant ˆ $\hat{G}_i$. On a la dŽcomposition
$$
\ES{D}= \ES{D}_1\times \cdots \times \ES{D}_d.
$$
Le groupe $G$ est un produit direct de $F$-groupes semisimples simplement connexes et $F$-quasi-simples, chacun d'eux correspondant ˆ une 
orbite de $\Gamma_F$ dans l'ensemble $\{\ES{D}_1,\ldots ,\ES{D}_d\}$. 

On suppose dans un premier temps que $G$ est $F$-quasi-simple, c'est-ˆ-dire que $\Gamma_F$ agit transitivement sur 
l'ensemble $\{\ES{D}_1,\ldots ,\ES{D}_d\}$. Soit $F_1/F$ l'extension sŽparable finie (de degrŽ $d$) telle que $\Gamma_{F_1}$ soit le stabilisateur 
de $\ES{D}_1$ dans $\Gamma_F$. Alors \cite[6.21]{BT} $G_1$ est dŽfini sur $F_1$ et $G\simeq R_{F_1/F}(G_1)$ o $R_{F_1/F}$ est le foncteur restriction des scalaires 
ˆ la Weil. Reprenons brivement la description de ${^LG}$ telle qu'elle est donnŽe dans \cite[5.1]{Bor}. Pour un groupe $A$, un sous-groupe d'indice fini $A_1\subset A$, et un groupe $V_1$ 
muni d'une structure de $A_1$-module (ˆ gauche), on note $V=I_{A_1}^A(V_1)$ l'ensemble des fonctions $v: A\rightarrow V_1$ telles $v(a_1 a)= a_1\cdot a(v)$ pour tout $(a_1,a)\in A_1\times A$. 
On munit $V$ de la structure de groupe donnŽe par $(vv')(a)=v(a)v'(a)$, et de la structure de $A$-module donnŽe par 
$(a\cdot v)(a')= v(a'a)$. Pour $a\in A$, on note $V_a$ le sous-groupe de $V$ formŽ des fonctions $v$ ˆ support dans $A_1a$. 
On a la dŽcomposition $V= \prod_{a\in A_1\backslash A}V_a$ et les groupes $V_a$ sont permutŽs par $A$. Notons $e$ l'ŽlŽment neutre de $A$. 
Le groupe $V_e$ est un sous-$A_1$-module de $V$, et 
l'application $V\rightarrow V_1,\, v\mapsto v(e)$ est un homomorphisme de $A_1$-modules qui induit un isomorphisme de $V_e$ sur $V_1$. Appliquons cela 
ˆ $A= \Gamma_F$ et $A_1= \Gamma_{F_1}$, et ˆ $V_1= \hat{G}_1$ muni de l'action galoisienne $\tau \mapsto \tau_{G_1}$ de $\Gamma_{F_1}$. Le groupe $\hat{G}$  s'identifie \`a 
$I_{\Gamma_{F_1}}^{\Gamma_F}(\hat{G}_1)$ muni de l'action galoisienne $\sigma \mapsto \sigma_G$ de $\Gamma_F$ donnŽe par
$$
\sigma_G(g)(\sigma')= g(\sigma'\! \sigma),\quad \sigma'\in \Gamma_F.
$$
La paire de Borel $(\hat{B},\hat{T})$ de $\hat{G}$ est Žgale ˆ $(I_{\Gamma_{F_1}}^{\Gamma_F}(\hat{B}_1), I_{\Gamma_{F_1}}^{\Gamma_F}(\hat{T}_1))$ pour une 
paire de Borel $(\hat{B}_1,\hat{T}_1)$ de $\hat{G}_1$ bien dŽterminŽe, et la paire $(\hat{B},\hat{T})$ est conservŽe par l'action de $\Gamma_F$ si et seulement si 
la paire $(\hat{B}_1,\hat{T}_1)$ est conservŽe par l'action de $\Gamma_{F_1}$, ce que l'on suppose. Enfin puisque l'homomorphisme naturel $W_F \rightarrow \Gamma_F$ induit une 
bijection $W_{F_1}\backslash W_F \rightarrow \Gamma_{F_1} \backslash \Gamma_F$, l'application $g \mapsto g \circ (W_F \rightarrow \Gamma_F)$ 
induit un isomorphisme $ I_{\Gamma_{F_1}}^{\Gamma_F}(\hat{G}_1)\buildrel \simeq \over{\longrightarrow } I_{W_{F_1}}^{W_F}(\hat{G}_1)$ dont 
l'inverse est $W_F$-Žquivariant. On pose
$$
{^LG}=I_{W_{F_1}}^{W_{F}}(\hat{G}_1) \rtimes W_F =  I_{\Gamma_{F_1}}^{\Gamma_F}(\hat{G}_1)\rtimes W_F.
$$
Posons $W= N_{\hat{G}}(\hat{T})/\hat{T}$ et $W_1= N_{\hat{G}_1}(\hat{T}_1)/\hat{T}_1$. 
On a l'ŽgalitŽ $W= I_{\Gamma_{F_1}}^{\Gamma_F}(W_1)$. 

Soit $\bs{G}'= (G'\!,\ES{G}'\!,s)$ une donnŽe endoscopique de $G$, avec $s\in \hat{T}$. On lui associe comme suit une donnŽe endoscopique
$$
\bs{G}'_{F_1}=(G'_1,\ES{G}'_1,s_1)
$$
de $G_1$. Posons $\hat{G}'=Z_{\hat{G}}(s)^\circ$ et $\hat{B}'= \hat{B} \cap \hat{G}'$. Pour chaque $w\in W_F$ d'image $\sigma$ dans $\Gamma_F$, choisissons un ŽlŽment $(g(w),w)\in \ES{G}'$ 
tel que ${\rm Int}_{g(w)}\circ \sigma_G$ conserve la paire de Borel $(\hat{B}'\!,\hat{T})$ de $\hat{G}'$. La classe $\hat{T}g(w)$ est uniquement dŽterminŽe. Il existe donc 
un ŽlŽment bien dŽterminŽ $\eta_{G'}(\sigma)$ de $W$ tel que l'action par conjugaison de $(g(w),w)$ sur $\hat{T}$ soit Žgale ˆ
$$
\sigma_{G'}= \eta_{G'}(\sigma)\sigma_G;
$$ 
o l'on identifie $\eta_{G'}(\sigma)$ ˆ l'automorphisme ${\rm Int}_{\eta_{G'}(\sigma)}$ de $\hat{T}$. L'ŽlŽment $s$ est fixŽ par 
l'action $\sigma\mapsto \sigma_{G'}$. L'application $\sigma \mapsto \eta_{G'}(\sigma)$ est un $1$-cocycle de $\Gamma_F$ ˆ valeurs dans $W$. 
Pour $\tau\in \Gamma_{F_1}$, posons\footnote{La notation n'est pas trs heureuse: rappelons que pour $\sigma\in \Gamma_F$, $\eta_{G'}(\sigma)$ est un ŽlŽment de 
$W= \prod_{\sigma'\in \Gamma_{F_1}\backslash \Gamma_F}W_{\sigma'}$, et que pour $\sigma'\in \Gamma_F$, $\eta_{G'}(\sigma)(\sigma')$ est la composante de $\eta_{G'}(\sigma)$ sur $W_{\sigma'}$.} $\eta_{G'\!,1}(\tau)= \eta_{G'}(\tau)(1)$. L'application $\tau \mapsto \eta_{G'\!,1}(\tau)$ est un $1$-cocycle 
de $\Gamma_{F_1}$ ˆ valeurs dans $W_1$. Notons $\tau \mapsto \tau_{G'\!,1}$ l'action de $\Gamma_{F_1}$ sur $\hat{T}_1$ donnŽe par
$$
\tau_{G'\!,1}= \eta_{G'\!,1}(\tau)\tau_{G_1}.
$$ 
On prend pour $s_1$ l'ŽlŽment $s(1)\in \hat{T}_1$. Il est fixŽ par l'action galoisienne $\tau\mapsto \tau_{G'\!,1}$ sur $\hat{T}_1$. Posons $\hat{G}'_1= Z_{\hat{G}_1}(s_1)^\circ$ et $\hat{B}'_1 = \hat{B}_1 \cap \hat{G}'_1$. Pour chaque $w_1\in W_{F_1}$ d'image $\tau$ dans $\Gamma_{F_1}$, choisissons un reprŽsentant 
$\tilde{\eta}_{G'\!,1}(w_1)=\tilde{\eta}_{G'\!,1}(\tau)$ de $\eta_{G'\!,1}(\tau)$ dans $N_{\hat{G}_1}(\hat{T}_1)$. L'automorphisme 
${\rm Int}_{\tilde{\eta}_{G'\!,1}(\tau)}\circ \tau_{G_1}$ conserve la paire de Borel 
$(\hat{B}'_1,\hat{T}_1)$ de $\hat{G}'_1$. Soit 
$$
\ES{G}'_1 = \{(g'_1\tilde{\eta}_{G'\!,1}(w_1),w_1): g'_1\in \hat{G}'_1,\, w_1\in W_{F_1}\}\subset {^L(G_1)}.
$$
L'ensemble $\ES{G}'_1$ est un groupe qui normalise $\hat{G}'_1$. On en dŽduit de la manire habituelle une $L$-action de $\Gamma_{F_1}$ sur $\hat{G}'_1$. Soit $G'_1$ 
un $F_1$-groupe rŽductif connexe quasi-dŽployŽ dont 
$\hat{G}'_1$, muni de cette action galoisienne, est le groupe dual.  Le triplet 
$\bs{G}'_{F_1}=(G'_1,\ES{G}'_1,s_1)$ est une donnŽe endoscopique de $G_1$. 

\begin{monlem}
L'application $\bs{G}'\mapsto \bs{G}'_{F_1}$ induit une bijection de l'ensemble des classes d'Žquivalence de donnŽes endoscopiques de $G$ sur 
l'ensemble des classes d'Žquivalence de donnŽes endoscopiques de $G_1$.
\end{monlem}

\begin{proof}
RŽciproquement, soit $\bs{G}'_1= (G'_1,\ES{G}'_1,s_1)$ une donnŽe endoscopique de $G_1$, avec $s_1\in \hat{T}_1$. On lui associe comme suit une donnŽe endoscopique 
$$
{\rm Res}_{F_1/F}(\bs{G}'_1)= (G'\!,\ES{G}'\!,s)
$$
de $G$. Posons $\hat{G}'_1= Z_{\hat{G}_1}(s_1)^\circ$ et $\hat{B}'_1= \hat{B}_1 \cap \hat{G}'_1$. 
Pour chaque $w_1\in W_{F_1}$ d'image $\tau$ dans $\Gamma_{F_1}$, choisissons un ŽlŽment $(g_1(w_1),w_1)\in \ES{G}'_1$ tel que l'automorphisme
${\rm Int}_{g_1(w_1)\circ }\tau_{G_1}$ conserve la paire de Borel $(\hat{B}'_1,\hat{T}_1)$ de $\hat{G}'_1$. La classe $\hat{T}_1 g_1(w_1)$ est bien dŽterminŽe. Il 
existe donc un ŽlŽment bien dŽterminŽ $\eta_{G'_1}(\tau)$ de $W_1$ tel que l'action par conjugaison de $(g_1(w_1),w_1)$ sur $\hat{T}_1$ soit donnŽe par
$$
\tau_{G'_1}= \eta_{G'_1}(\tau) \tau_{G_1}.
$$
L'application $\tau \mapsto \eta_{G'_1}(\tau)$ est un $1$-cocycle. Notons $\alpha_1\in {\rm H}^1(\Gamma_{F_1}, W_1)$ sa classe de cohomologie, et 
soit $\alpha\in {\rm H}^1(\Gamma_F,W)$ l'image de $\alpha_1$ par l'inverse de l'isomorphisme de Shapiro
$$
{\rm H}^1(\Gamma_F,I_{\Gamma_{F_1}}^{\Gamma_F}(W_1))\buildrel \simeq \over{\longrightarrow} {\rm H}^1(\Gamma_{F_1},W_1),\,\beta \mapsto 
(\tau \mapsto \beta(\tau)(1)).
$$
Choisissons un $1$-cocycle $\sigma \mapsto \eta_{G'_1,G}(\sigma)$ de $\Gamma_F$ ˆ valeurs dans $W$ qui soit dans la classe de cohomologie $\alpha$. 
Quitte ˆ remplacer $\eta_{G'_1,G}$ par un $1$-cocycle cohomologue, on peut supposer, et on suppose, que 
$\eta_{G'_1,G}(\tau)(1)= \eta_{G'_1}(\tau)$ pour tout $\tau\in \Gamma_{F_1}$. Notons $\sigma \mapsto \sigma_{G'_1,G}$ l'action de $\Gamma_F$ sur $\hat{T}$ donnŽe par
$$
\sigma_{G'_1,G}= \eta_{G'_1,G}(\sigma)\sigma_{G}.
$$
Soit $\Gamma_F \rightarrow \hat{T}_1,\, \sigma \mapsto s(\sigma)$ l'application dŽfinie par
$$
s(\sigma)= \eta_{G'_1,G}(\sigma)(1)^{-1}(s_1).
$$
Pour $\tau\in \Gamma_{F_1}$ et $\sigma\in \Gamma_F$, on a
\begin{eqnarray*}
s(\tau \sigma)&= &[\eta_{G'_1,G}(\tau)\tau_G(\eta_{G'_1,G}(\sigma))](1)^{-1}(s_1)\\
&=& \tau_{G}(\eta_{G'_1,G}(\sigma))(1)^{-1}\eta_{G'_1,G}(\tau)(1)^{-1}(s_1).
\end{eqnarray*}
Or
$$
\tau_G(\eta_{G'_1,G}(\sigma))(1)= \eta_{G'_1,G}(\sigma)(\tau)=\tau_{G_1}(\eta_{G'_1,G}(\sigma)(1)),
$$
d'o
\begin{eqnarray*}
s(\tau \sigma)&= &\tau_{G_1}[\eta_{G'_1,G}(\sigma)(1)^{-1} (\eta_{G'_1}(\tau)\tau_{G_1})^{-1}(s_1)]\\
& = & \tau_{G_1}(s(\sigma))
\end{eqnarray*}
car $s_1$ est fixŽ par l'action $\tau \mapsto \tau_{G'_1}=\eta_{G'_1}(\tau)\tau_{G_1}$ de $\Gamma_{F_1}$. Par consŽquent l'application 
$$
\Gamma_F\rightarrow \hat{T},\, \sigmaÊ\mapsto s(\sigma)
$$
est un ŽlŽment de $\hat{T}$, qui vŽrifie $s(1)=s_1$. On vŽrifie facilement que 
$s$ est fixŽ par l'action $\sigma \mapsto \sigma_{G'_1,G}(\sigma)$ de $\Gamma_F$. Posons $\hat{G}'= Z_{\hat{G}}(s)^\circ$ et 
$\hat{B}'= \hat{G}' \cap \hat{B}$. Pour chaque $w\in W_F$ d'image $\sigma$ dans $\Gamma_F$, choisissons un reprŽsentant $\tilde{\eta}_{G'_1,G}(w)= \tilde{\eta}_{G'_1,G}(\sigma)$ de 
$\eta_{G'_1,G}(\sigma)$ dans $N_{\hat{G}}(\hat{T})$. L'auto\-morphisme 
${\rm Int}_{\tilde{\eta}_{G'_1,G}(\sigma)}\circ \sigma_{G}$ conserve la paire de Borel 
$(\hat{B}'\!,\hat{T})$ de $\hat{G}'$. Soit 
$$
\ES{G}' = \{(g'\tilde{\eta}_{G'_1,G}(w),w): g'\in \hat{G}',\, w\in W_{F}\}\subset {^LG}.
$$
L'ensemble $\ES{G}'$ est un groupe qui normalise $\hat{G}'$. On en dŽduit une $L$-action de $\Gamma_F$ sur $\hat{G}'$. Soit $G'$ 
un $F$-groupe rŽductif connexe quasi-dŽployŽ dont 
$\hat{G}'$, muni de cette action galoisienne, est le groupe dual. 
Le triplet ${\rm Res}_{F/F_1}(\bs{G}'_1)=(G'\!,\ES{G}'\!,s)$ est une donnŽe endoscopique de $G$, bien dŽterminŽe ˆ Žquivalence prs par 
un ŽlŽment $x$ de $N_{\hat{G}}(\hat{T})$ tel que 
$x(1)\in {\rm Aut}(\bs{G}'_1)$ et ${\rm Int}_{\eta_{G'_1}(\tau)} \circ \tau_{G_1}(\bar{x}(1))= \bar{x}(1)$ pour tout $\tau\in W_{F_1}$; o 
$\bar{x}$ est l'image de $x$ dans $W$. Par construction, on a
$$
({\rm Res}_{F/F_1}(\bs{G}'_1))_{F_1}= \bs{G}'_1.
$$
Cela prouve le lemme.
\end{proof}

Soit $\bs{G}''$ une autre donnŽe endoscopique de $G$. Si $x\in \hat{G}$ est une \'equivalence entre $\bs{G}'$ et $\bs{G}''$, alors $x(1)$  est une Žquivalence entre $\bs{G}'_{F_1}$ et $\bs{G}''_{F_1}$. RŽciproquement, si $x_{1}\in \hat{G}_1$ est une Žquivalence entre $\bs{G}'_{F_1}$ et $\bs{G}''_{F_1}$, alors on  peut construire un $x\in \hat{G}$  qui est une Žquivalence entre $\bs{G}'$ et $\bs{G}''$ et tel que $x(1)=x_{1}$. 

 Comme la propriŽtŽ d'tre $F$-quasi-simple n'est pas conservŽe par localisation, il nous faut maintenant revenir au cas gŽnŽral. \'Ecrivons 
$\ES{D} =\bigcup_{k=1}^r (\ES{D}_{k,1}\cup \cdots \cup \ES{D}_{k,d_k})$ o, pour chaque $k$, $\ES{D}_{k,1}\cup \cdots \cup \ES{D}_{k,d_k}$ est une orbite dans l'ensemble des 
composantes connexes de $\ES{D}$. Pour $k=1,\ldots ,r$, soit $F_k/F$ l'extension sŽparable finie telle que $
\Gamma_{F_k}$ soit le stabilisateur de $\ES{D}_{k,1}$ dans $\Gamma_F$. Le groupe $G_k$ correspondant ˆ $\ES{D}_{k,1}$ est dŽfini sur $F_k$, et l'on a 
$G \simeq \prod_{k=1}^r  H_{k}$ o\`u $H_k = {\rm Res}_{F_k/F}(G_k)$. Pour $k=1,\ldots ,r$, soit $\bs{H}'_{\!k}= (H'_k,\ES{H}'_k,s_k)$ une donnŽe endoscopique de $H_k$. La famille 
des $\bs{H}'_{\!k}$ dŽfinit une donnŽe endoscopique $\bs{G}' =(G'\!,\ES{G}'\!,s)$ de $G$: on prend $s=(s_1,\ldots ,s_r)$, $G'=H'_1\times \cdots \times H'_r$, et 
$\ES{G}'$ est l'ensemble des $((h_1,\ldots ,h_r),\sigma)\in {^LG}$ tels que pour chaque $k$, $(h_k,\sigma)$ appartient ˆ $\ES{H}'_k$. Toute donnŽe endoscopique de $\bs{G}$ est obtenue de cette manire. Par abus d'Žcriture, on note $\bs{G}'= \prod_{k=1}^r\bs{H}'_{\!k}$. \`A chaque donnŽe $\bs{H}'_{\!k}$ est associŽe comme ci-dessus une donnŽe 
endoscopique $\bs{G}'_{k}= (\bs{H}'_{\!k})_{F_k}$ de $G_k$. L'application qui ˆ $\bs{G}'$ associe la famille $\{\bs{G}'_1,\ldots ,\bs{G}'_k\}$ est compatible ˆ la relation d'Žquivalence entre donnŽes endoscopiques en un sens Žvident (dŽduit par produit du cas $r=1$). 

Si $F$ est global, pour chaque place $v$ de $F$, chaque $k$, et chaque place $w$ de $F_k$ au-dessus de $v$, le complŽtŽ $F_{k,w}$ de $F_k$ en $w$ est une extension 
sŽparable finie de $F_v$. Le groupe $G_v= G\times_F F_v$ est isomorphe ˆ $\prod_{k=1}^r H_{k,v}$ avec 
$$
H_{k,v}= {\rm Res}_{F_k/F}(G_k)\times_F F_v \simeq  \prod_{w_k\vert v} {\rm Res}_{F_{k,w_k}/F_v}(G_{k,w_k})
$$
o $w_k$ parcourt les places de $F_k$ au-dessus de $v$, et $G_{k,w_k}= G_k \times_{F_k}F_{k,w_k}$. 
C'est donc encore un groupe du mme type, c'est-ˆ-dire un produit fini de groupes $F_v$-quasi-simples 
simplement connexes. L'application qui ˆ une donnŽe endoscopique $\bs{G}'$ de $G$ associe la famille $\{\bs{G}'_1,\ldots, \bs{G}'_r\}$ est compatible ˆ la localisation 
au sens suivant. \'Ecrivons $\bs{G}'= \prod_{k=1}^r \bs{H}'_{\!k}$. Pour chaque place $v$ de $F$, et pour chaque $k$, la donnŽe endoscopique $\bs{H}'_{\!k,v}$ de $H_{k,v}$ obtenue par 
localisation de $\bs{H}'_{\!k}$ en $v$ se dŽcompose (avec le mme abus d'Žcriture que plus haut) en
$$
\bs{H}'_{\!k,v} = \prod_{w_k\vert v} \bs{H}'_{\!k,w_k}
$$
pour des donnŽes endoscopiques $\bs{H}'_{\!k,w_k}$ de $H_{k,w_k}={\rm Res}_{F_{k,w_k}/F_v}(G_{k,w_k})$. D'autre part pour chaque place $w_k$ de $F_k$ au-dessus de $v$, la donnŽe endoscopique $\bs{G}'_k = (\bs{H}'_{\!k})_{F_k}$ de $G_k$ donne par localisation une donnŽe endoscopique $\bs{G}'_{k,w_k}$ de $G_{k,w_k}$. La compatibiltŽ en question est (pour tout $v$, tout $k$, et tout $w_k\vert v$)
$$
(\bs{H}'_{\!k,w_k})_{F_{k,w_k}}= \bs{G}'_{k,w_k}.
$$  
On en dŽduit qu'il suffit de prouver la proposition de \ref{le rŽsultat} dans le cas $r=1$ et $d_1=1$, 
\cad le cas o $\hat{G}\;(=\hat{G}_{\rm AD})$ est simple. 

\subsection{Hypothses et dŽfinitions}\label{hypothses et dŽfinitions}
Continuons avec les notations de \ref{le rŽsultat} et \ref{rŽduction}. Sauf prŽcision, le corps $F$ est global ou local. On suppose jusqu'ˆ la fin de l'article que 
$G$ est quasi-simple et simplement connexe, \cad que $\hat{G}= \hat{G}_{\rm AD}$ et $\ES{D}$ est connexe. Pour allŽger l'Žcriture, nous utilisons dŽsormais la forme 
\og galoisienne\fg des $L$-groupes: $^LG=\hat{G}\rtimes \Gamma_{F}$. L'adaptation de ce qui suit  \`a la forme \og groupes de Weil\fg ne pose aucun probl\`eme. 

Notons $\alpha_{0}$ l'oppos\'ee de la plus grande racine dans $\Sigma^+$ et $\Delta_{\rm a}=\{\alpha_{0}\}\cup \Delta$. En \'ecrivant $-\alpha_{0}$ dans la base $\Delta$, on obtient une relation
$$
\sum_{\alpha\in \Delta_{\rm a}}d(\alpha)\alpha=0\leqno{(1)}
$$
o\`u $d(\alpha_{0})=1$ et, pour $\alpha\in \Delta$, $d(\alpha)$ appartient \`a l'ensemble ${\mathbb N}_{>0}$ des entiers strictement positifs. On sait que   l'espace des relations entre les \'el\'ements de $\Delta_{\rm a}$ est la droite port\'ee par la relation (1). Cela implique que
\begin{enumerate}[leftmargin=17pt]
\item[(2)]la relation (1) est la seule relation lin\'eaire entre les \'el\'ements de $\Delta_{\rm a}$ dont les coefficients sont entiers relatifs et dont au moins un coefficient vaut $1$.
\end{enumerate}
Rappelons aussi la propri\'et\'e suivante. Pour $\beta\in \Sigma^+$, \'ecrivons $\beta=\sum_{\alpha\in \Delta}m(\alpha)\alpha$ avec  $m(\alpha)\in {\mathbb N}$ pour tout $\alpha\in \Delta$. Alors 
\begin{enumerate}[leftmargin=17pt]
\item[(3)]$m(\alpha)\leq d(\alpha)$ pour tout $\alpha\in \Delta$.
\end{enumerate}

Notons $\ES{D}_{\rm a}$ le diagramme de Dynkin compl\'et\'e dont l'ensemble de sommets est $\Delta_{\rm a}$. Notons ${\rm Aut}(\ES{D})\subset {\rm Aut}(\ES{D}_{\rm a})$ les groupes d'automorphismes de $\ES{D}$ et $\ES{D}_{\rm a}$.  Remarquons que, puisque $\hat{G}$ est adjoint, ${\rm Aut}(\ES{D}_{\rm a})$ se plonge naturellement dans le groupe d'automorphismes de $\hat{T}$. Puisque la paire $(\hat{B},\hat{T})$ est conserv\'ee par l'action galoisienne, l'action galoisienne se restreint en une action sur $\Delta$ et $\Delta_{\rm a}$, et peut \^etre consid\'er\'ee comme un homomorphisme $\Gamma_{F}\to {\rm Aut}(\ES{D})$. D'ailleurs comme on le sait, la donn\'ee de l'action galoisienne est \'equivalente \`a celle de cet homomorphisme. On note $E$ l'extension galoisienne finie de $F$ telle que $\Gamma_{E}$ soit le noyau de cet homomorphisme. 

Rappelons qu'on a notŽ $W$ le groupe de Weyl de $\hat{G}$ relatif \`a $\hat{T}$. Soit $\Omega$ le sous-groupe des \'el\'ements de $W$ qui conservent $\Delta_{\rm a}$. L'application qui \`a un \'el\'ement de $\Omega$ associe son action sur $\ES{D}_{\rm a}$ identifie $\Omega$ \`a un sous-groupe de ${\rm Aut}(\ES{D}_{\rm a})$. On sait \cite{Bou}[VI.4.3] 
que c'est un sous-groupe ab\'elien distingu\'e de ${\rm Aut}(\ES{D}_{\rm a})$, et que
$$
{\rm Aut}(\ES{D}_{\rm a})= \Omega \rtimes {\rm Aut}(\ES{D}).
$$
Signalons que, pour $\tau\in {\rm Aut}(\ES{D}_{\rm a})$, 
on a $d(\tau(\alpha))=d(\alpha)$ pour toute racine $\alpha\in \Delta_{\rm a}$:  en appliquant $\tau^{-1}$ \`a (1), on obtient
$\sum_{\alpha\in \Delta_{\rm a}}d(\tau(\alpha))\alpha=0$; c'est une relation entre les \'el\'ements de $\Delta_{\rm a}$  qui v\'erifie les conditions de (2), donc c'est la relation (1). On voit en inspectant chaque syst\`eme de racines que:
\begin{enumerate}[leftmargin=17pt]
\item[(4)]l'application $\omega \mapsto \omega (\alpha_{0})$ est une bijection de $\Omega$ sur l'ensemble des racines $\alpha\in \Delta_{\rm a}$ telles que $d(\alpha)=1$.
\end{enumerate}
%

\subsection{La construction de Langlands}\label{la construction de Langlands}
On fixe  pour tout $n\in  {\mathbb N}_{>0}$ un \'el\'ement $\zeta_{n}\in {\mathbb C}^{\times}$ qui est une racine primitive de $1$ d'ordre $n$. Soit $\bs{G}'=(G'\!,\ES{G}'\!,s)$ une donn\'ee endoscopique de $G$ (rappelons que $\ES{G}'\subset {^LG}= \hat{G} \rtimes \Gamma_F$). \`A \'equivalence pr\`es, on peut supposer, et on suppose, $s\in \hat{T}$. Dans ce paragraphe \ref{la construction de Langlands}, on impose 
la condition:
\begin{enumerate}[leftmargin=17pt]
\item[(1)]$s$ est d'ordre fini.
\end{enumerate}
On note $d$ cet ordre.

Nous allons rappeler les r\'esultats de Langlands \cite[pp.~708-709]{L}. Commenons par reprendre la construction habituelle 
(cf. \ref{rŽduction}). On pose $\hat{G'}= Z_{\hat{G}}(s)^\circ \supset \hat{T}$, et on fixe un sous-groupe de Borel $\hat{B}'$ de $\hat{G}'$ contenant $\hat{T}$. Pour tout $\sigma\in \Gamma_{F}$, on peut choisir un \'el\'ement $(g(\sigma),\sigma)\in \ES{G}'$ dont l'action par conjugaison conserve  la paire $(\hat{B}',\hat{T})$. La classe $\hat{T}g(\sigma)$ est uniquement d\'etermin\'ee. Il existe donc un \'el\'ement bien d\'etermin\'e $w_{G'}(\sigma)\in W$ tel que l'action par conjugaison de $(g(\sigma),\sigma)$ sur $\hat{T}$ soit \'egale \`a $w_{G'}(\sigma)\sigma_{G}$. L'application
$$
\sigma\mapsto \sigma_{G'}=
w_{G'}(\sigma)\sigma_{G}
$$
est un homomorphisme de $\Gamma_{F}$ dans le groupe d'automorphismes de $\hat{T}$ qui fixent $s$. 

Pour $k\in \{0,\ldots,d-1\}$, posons $\mathfrak{Y}_{k}=\{\alpha\in \Sigma; \alpha(s)=\zeta_{d}^k\}$. L'ensemble $\mathfrak{Y}_{0}$ est celui des racines de $\hat{T}$ dans $\hat{G}'$. On note $\mathfrak{X}_{0}$ l'ensemble de racines simples relatif \`a $\hat{B}'$. On introduit dans $\Sigma$ la relation d'ordre partiel $\leq_{\hat{B}'}$ suivante: pour $\alpha,\,\beta\in \Sigma$, $\alpha \leq_{\hat{B}'}\beta$ si et seulement si $\beta-\alpha$ est \'egal \`a une combinaison lin\'eaire des \'el\'ements de $\mathfrak{X}_{0}$ \`a coefficients dans ${\mathbb N}$. Pour $k\in \{1,\ldots,d-1\}$, notons $\mathfrak{Z}_{k}$ l'ensemble des \'el\'ements de $\mathfrak{Y}_{k}$ qui ne sont pas combinaisons lin\'eaires d'\'el\'ements de $\bigcup_{j=0,\ldots,k-1}\mathfrak{Y}_{j}$ \`a coefficients dans ${\mathbb Z}$. En particulier $\mathfrak{Z}_{1}=\mathfrak{Y}_{1}$. On note $\mathfrak{X}_{k}$ l'ensemble des \'el\'ements de $\mathfrak{Z}_{k}$ qui sont minimaux pour l'ordre $\leq_{\hat{B}'}$. On pose $\mathfrak{X}=\bigcup_{k=0,\ldots,d-1}\mathfrak{X}_{k}$. Le r\'esultat de Langlands est:
\begin{enumerate}[leftmargin=17pt]
\item[(2)]il existe $u\in W$ tel que $u(\mathfrak{X})=\Delta$ ou $u(\mathfrak{X})=\Delta_{\rm a}$.
\end{enumerate}
Quitte \`a remplacer $\bs{G}'$ par la donn\'ee \'equivalente conjugu\'ee par un rel\`evement de $u$ dans le normalisateur $N_{\hat{G}}(\hat{T})$ de $\hat{T}$ dans $\hat{G}$ et \`a remplacer $\hat{B}'$ par $u(\hat{B}')$, on peut supposer $u=1$; donc $\mathfrak{X}=\Delta$ ou $\mathfrak{X}= \Delta_{\rm a}$. Remarquons que, jusque-l\`a, les constructions ne d\'ependent que de $s$ et d'un choix de $\hat{B}'$ mais   pas de $\ES{G}'$. 

Justement, puisque les constructions ne d\'ependent que de $s$ et de $\hat{B}'$, elles sont conserv\'ees par le groupe des automorphismes de $\hat{T}$ qui fixent $s$ et conservent $\hat{B}'$ (conserver $\hat{B}'$ revient \`a conserver l'ensemble $\mathfrak{X}_0$ de racines simples relatif ˆ $\hat{B}'$). En particulier, pour tout $\sigma\in \Gamma_{F}$, l'automorphisme $\sigma_{G'}=w_{G'}(\sigma)\sigma_{G}$ de $\hat{T}$ conserve chaque $\mathfrak{X}_{k}$ et $\mathfrak{X}$ tout entier. L'automorphisme $\sigma_{G}$ conservant  $\Delta$ et $\Delta_{a}$, donc $\mathfrak{X}$, $w_{G'}(\sigma)$ conserve aussi $\mathfrak{X}$. Supposons d'abord $\mathfrak{X}=\Delta$. Le seul \'el\'ement de $W$ qui conserve $\Delta$ est l'identit\'e. Donc $w_{G'}(\sigma)=1$.   Il en r\'esulte que
\begin{enumerate}[leftmargin=17pt]
\item[(3)]si $\mathfrak{X}=\Delta$, alors $\ES{G}'=\hat{G'}\rtimes \Gamma_{F}$.
\end{enumerate}
Supposons maintenant $\mathfrak{X}=\Delta_{a}$.  Puisque $w_{G'}(\sigma)$ conserve $\Delta_{a}$, on a $w_{G'}(\sigma)\in \Omega$. Dans la suite, $w_{G'}(\sigma)$ et $\sigma_{G'}$ seront la plupart du temps consid\'er\'es comme des automorphismes de $\ES{D}_{\rm a}$. 
Notons $K$ l'extension galoisienne finie  de $F$ telle que $\Gamma_{K}$ soit le noyau de l'homomorphisme $\sigma\mapsto \sigma_{G'}$. On a
\begin{enumerate}[leftmargin=17pt]
\item[(4)]$E\subset K$ et la restriction \`a $\Gamma_{E}$ de l'application $\sigma\mapsto w_{G'}(\sigma)$ se quotiente en  une injection de $\Gamma_{K/E}$ dans $\Omega$.
\end{enumerate}
En effet, pour $\sigma\in \Gamma_{K}$, on a $w_{G'}(\sigma)\sigma_{G}=1$. Puisque ${\rm Aut}(\ES{D}_{\rm a})$ est produit semi-direct de $\Omega$ et ${\rm Aut}(\ES{D})$, 
cette \'egalit\'e entra\^{\i}ne $w_{G'}(\sigma)=1$ et $\sigma_{G}=1$. L'assertion (4) s'en d\'eduit.

\subsection{Equivalences de donn\'ees endoscopiques}\label{Žquivalences de donnŽes endoscopiques}
 On consid\`ere maintenant deux donn\'ees endoscopiques $\bs{G}'_{1}=(G'_{1},\ES{G}'_{1},s)$ et $\bs{G}'_{2}=(G'_{2},\ES{G}'_{2},s)$ de $G$, dont l'\'el\'ement $s$ est le m\^eme. Les groupes $\hat{G}'_{1}$ et $\hat{G}'_{2}$ sont donc les m\^emes et on note simplement $\hat{G}'$ ce groupe. On suppose comme dans le paragraphe pr\'ec\'edent que $s$ est d'ordre fini $d$. On applique les constructions pr\'ec\'edentes aux deux donn\'ees, en affectant d'indices $1$ et $2$ les objets associ\'es \`a nos deux donn\'ees. On peut choisir le m\^eme groupe de Borel $\hat{B}'$ pour les deux constructions. Puisque celles-ci ne d\'ependent que de $s$ et de ce choix de Borel, on obtient les m\^emes ensembles $\mathfrak{X}_{k}$ et $\mathfrak{X}$ pour les deux donn\'ees. Les seules diff\'erences entre nos donn\'ees sont les actions galoisiennes $\sigma\mapsto \sigma_{G'_{1}}$ et $\sigma\mapsto \sigma_{G'_{2}}$. L'assertion \ref{la construction de Langlands}.(3) entra\^{\i}ne que
\begin{enumerate}[leftmargin=17pt]
\item[(1)]si $\mathfrak{X}=\Delta$, les deux donn\'ees sont \'equivalentes.
\end{enumerate}

 Supposons $\mathfrak{X}=\Delta_{\rm a}$. Montrons que
\begin{enumerate}[leftmargin=17pt]
\item[(2)]les donn\'ees $\bs{G}'_{1}$ et $\bs{G}'_{2}$ sont \'equivalentes si et seulement s'il existe $\omega\in \Omega$ tel que 
$\omega(\mathfrak{X}_{k})=\mathfrak{X}_{k}$ pour $k=1,\ldots,d-1$ et $\omega \sigma_{G'_{1}}\omega^{-1}=\sigma_{G'_{2}}$ pour tout $\sigma\in \Gamma_{F}$.
\end{enumerate}
Si un tel $\omega$ existe, on le rel\`eve en un \'el\'ement quelconque  $x\in N_{\hat{G}}(\hat{T})$ et on v\'erifie que $x$ est une \'equivalence entre les deux donn\'ees. Inversement, soit $x$ une telle \'equivalence qui conjugue $\bs{G}'_{1}$ en $\bs{G}'_{2}$. On peut multiplier $x$ \`a gauche ou \`a droite par un \'el\'ement de $\hat{G}'$ et supposer que la conjugaison ${\rm Int}_x$ conserve la paire de Borel $(\hat{B'},\hat{T})$. Par d\'efinition, ${\rm Int}_x$ fixe $s$. Notons $\omega$ l'image de ${\rm Int}_x$ dans $W$. Comme on l'a d\'ej\`a dit, un automorphisme de $\hat{T}$ qui fixe $s$ et conserve $\hat{B}'$ conserve tous les objets d\'efinis au paragraphe pr\'ec\'edent. Donc $\omega$ conserve $\mathfrak{X}=\Delta_{\rm a}$, c'est-\`a-dire $\omega\in \Omega$. On a $\omega(\mathfrak{X}_{k})=\mathfrak{X}_{k}$ pour $k=1,\ldots,d-1$ comme on vient de le dire. Soit $\sigma\in \Gamma_{F}$. Pour $i=1,\,2$, on rel\`eve $w_{G'_{i}}(\sigma)$ en un \'el\'ement $g_{i}(\sigma)\in N_{\hat{G}}(\hat{T})$. On a $(g_{i}(\sigma),\sigma)\in \ES{G}'_{i}$. Par d\'efinition d'une \'equivalence, on a $x(g_{1}(\sigma),\sigma)x^{-1}\in \ES{G}'_{2}$, donc $xg_{1}(\sigma)\sigma_{G}(x)^{-1}=hg_{2}(\sigma)$ pour un $h\in \hat{G}'$. Tous les \'el\'ements autres que $h$ normalisant $\hat{T}$, $h$ le normalise aussi. Puisque les conjugaisons par $x$, $(g_{1}(\sigma),\sigma)$ et $(g_{2}(\sigma),\sigma)$ conservent le  groupe $\hat{B}'$, la conjugaison par $h$ le conserve aussi. Un \'el\'ement $y$ de $\hat{G}'$ tel que ${\rm Int}_y$ conserve $(\hat{B}',\hat{T})$ appartient \`a $\hat{T}$, donc $h\in \hat{T}$. En projetant dans $W$ l'\'egalit\'e    $xg_{1}(\sigma)\sigma_{G}(x)^{-1}=hg_{2}(\sigma)$, on obtient $\omega \sigma_{G'_{1}}\omega^{-1}=\sigma_{G'_{2}}$. Cela d\'emontre (2). 

\vskip1mm
Consid\'erons  une unique donn\'ee $\bs{G}'=(G'\!,\ES{G}'\!,s)$. On note ${\rm Out}(\bs{G}')$ le groupe des composantes connexes du groupe d'automorphismes de $\bs{G}'$. Si  $s$ est d'ordre fini et $\mathfrak{X}=\Delta_{\rm a}$,
  on d\'emontre de la m\^eme fa\c{c}on que
\begin{enumerate}[leftmargin=17pt]
\item[(3)]${\rm Out}(\bs{G}')$ est le groupe des $\omega\in \Omega$ tels que $\omega(\mathfrak{X}_{k})=\mathfrak{X}_{k}$ pour $k=1,\ldots,d-1$ et $\omega \sigma_{G'}\omega^{-1}=\sigma_{G'}$ pour tout $\sigma\in \Gamma_{F}$.
\end{enumerate}
%

\subsection{\'Equivalence presque partout, cas particulier}\label{cas particulier}
On suppose dans ce paragraphe et le suivant que $F$ est un corps de nombres. 
Soit $\bs{G}'=(G'\!,\ES{G}'\!,s)$ une donn\'ee endoscopique de $G$. Rappelons que pour chaque place $v$ de $F$, 
on en d\'eduit une donn\'ee endoscopique $\bs{G}'=(G'_{v},\ES{G}'_{v},s)$ de $G_{v}=G\times_F F_v$, cf. \ref{le rŽsultat}. 
En particulier, $\ES{G}'_{v}$ est le sous-groupe des $(g,\sigma)\in \ES{G}'$ tels que $\sigma\in \Gamma_{v}$.

Consid\'erons deux  donn\'ees endoscopiques $\bs{G}'_{1}=(G'_{1},\ES{G}'_{1},s_1)$ et $\bs{G}'_{2}=(G'_{2},\ES{G}'_{2},s_2)$ de $G$. On suppose:
\begin{enumerate}[leftmargin=17pt]
\item[(1)] les ŽlŽments $s_{1}$ et $s_{2}$ sont d'ordre fini;
\vskip1mm
\item[(2)] les donn\'ees $\bs{G}'_{1,v}$ et $\bs{G}'_{2,v}$ sont \'equivalentes pour presque toute place $v$.
\end{enumerate}
%

\begin{mapropo}
Sous ces hypoth\`eses, la proposition de \ref{le rŽsultat} est vraie: les donnŽes $\bs{G}'_{1}$ et $\bs{G}'_{2}$ sont \'equivalentes.
\end{mapropo}

\begin{proof}
L'\'equivalence en au moins une place entra\^{\i}ne que $s_{1}$ et $s_{2}$ sont conjugu\'es. Quitte \`a remplacer l'une de nos donn\'ees par une donn\'ee \'equivalente, on peut supposer $s_{1}=s_{2}$ et on note simplement $s$ cet \'el\'ement. On applique les constructions des paragraphes pr\'ec\'edents, on obtient en particulier un ensemble $\mathfrak{X}$ commun qui est \'egal soit \`a $\Delta$ soit \`a $\Delta_{\rm a}$. Si $\mathfrak{X}=\Delta$, l'assertion \ref{Žquivalences de donnŽes endoscopiques}.(1) entra\^{\i}ne que $\bs{G}'_{1}$ et $\bs{G}'_{2}$ sont \'equivalentes. Supposons $\mathfrak{X}=\Delta_{\rm a}$. Notons $\Omega^{\sharp}$ le sous-groupe des $\omega\in \Omega$ tels que $\omega(\mathfrak{X}_{k})=\mathfrak{X}_{k}$ pour tout $k=1,\ldots,d-1$. Montrons que
\begin{enumerate}[leftmargin=17pt]
\item[(3)]pour tout $\sigma\in \Gamma_{F}$, il existe $\omega\in \Omega^{\sharp}$ tel que $\omega \sigma_{G'_{1}}\omega^{-1}=\sigma_{G'_{2}}$.
\end{enumerate} 
D'apr\`es le th\'eor\`eme de Tchebotarev et l'hypoth\`ese (2), il existe une place $v$ et un \'el\'ement $\gamma\in \Gamma_{F}$ tels que les donnŽes 
$\bs{G}'_{1,v}$ et $\bs{G}'_{2,v}$ soient \'equivalentes et $\gamma\sigma\gamma^{-1}\in \Gamma_{F_{v}}$. En appliquant \ref{Žquivalences de donnŽes endoscopiques}.(2) aux donn\'ees locales, il existe $\omega'\in \Omega^{\sharp}$ tel que
$$
\omega'\gamma_{G'_{1}}\sigma_{G'_{1}}\gamma_{G'_{1}}^{-1}{\omega'}^{-1}=\gamma_{G'_{2}}\sigma_{G'_{2}}\gamma_{G'_{2}}^{-1}.
$$
Posons $\omega=\gamma_{G'_{2}}^{-1}\omega' \gamma_{G'_{1}}=\gamma_{G}^{-1}\omega_{G'_{2}}(\gamma)^{-1}\omega'\omega_{G'_{1}}(\gamma)\gamma_{G}$. Puisque $\Omega$ est distingu\'e dans ${\rm Aut}(\ES{D}_{\rm a})$, on a $\omega\in \Omega$. Puisque $\omega'$ et $\gamma_{G'_{2}}$ conservent tous deux les ensembles $\mathfrak{X}_{k}$, $\omega$ les conserve aussi, donc $\omega\in \Omega^{\sharp}$. Enfin, on a bien $\omega \sigma_{G'_{1}}\omega^{-1}=\sigma_{G'_{2}}$. Cela d\'emontre (3).

On a
\begin{enumerate}[leftmargin=17pt]
\item[(4)]$w_{G'_{1}}(\sigma)=w_{G'_{2}}(\sigma)$ pour tout $\sigma\in \Gamma_{E}$.
\end{enumerate}
En effet, pour $\sigma\in \Gamma_{E}$, on a simplement $\sigma_{G'_{i}}=w_{G'_{i}}(\sigma)$ pour $i=1,\,2$. Ces \'el\'ements $w_{G'_{i}}(\sigma)$ appartiennent \`a $\Omega$ et sont conjugu\'es par un \'el\'ement de $\Omega$ d'apr\`es (3). Puisque $\Omega$ est commutatif, ils sont \'egaux.

Si $E=F$, l'assertion (4) entra\^{\i}ne que $\bs{G}'_{1}=\bs{G}'_{2}$. Supposons maintenant que $\Gamma_{E/F}$ soit cyclique et fixons un \'el\'ement $\gamma\in \Gamma_{F}$ dont l'image dans $\Gamma_{E/F}$ soit un g\'en\'erateur de ce groupe. Appliquons (3) et fixons $\omega\in \Omega^{\sharp}$ tel que $\omega\gamma_{G'_{1}}\omega^{-1}=\gamma_{G'_{2}}$. Consid\'erons les deux homomorphismes $\sigma\mapsto \sigma_{G'_{2}}$ et $\sigma\mapsto \omega\sigma_{G'_{1}}\omega^{-1}$ de $\Gamma_{F}$ dans ${\rm Aut}(\ES{D}_{\rm a})$. L'ensemble des $\sigma\in \Gamma_{F}$ en lesquels ils co\"{\i}ncident est un sous-groupe de $\Gamma_{F}$. Ce sous-groupe contient $\Gamma_{E}$ d'apr\`es (4) et parce que $\Omega$ est commutatif. Il contient aussi $\gamma$ par le choix de $\omega$. Mais le sous-groupe engendr\'e par $\Gamma_{E}$ et $\gamma$ est $\Gamma_{F}$ tout entier. Donc $\sigma_{G'_{2}} = \omega\sigma_{G'_{1}}\omega^{-1}$  pour tout $\sigma\in \Gamma_{F}$. D'apr\`es \ref{Žquivalences de donnŽes endoscopiques}.(2), nos deux donn\'ees endoscopiques sont donc \'equivalentes. 

Remarquons que l'application $\sigma\mapsto \sigma_{G}$ se quotiente en une injection de $\Gamma_{E/F}$ dans ${\rm Aut}(\ES{D})$. Donc l'hypoth\`ese que $\Gamma_{E/F}$ est cyclique est  v\'erifi\'ee sauf dans le cas o\`u $G$ est de type $D_{4}$ et o\`u $\sigma\mapsto \sigma_{G}$ se quotiente en  un isomorphisme de $\Gamma_{E/F}$ sur ${\rm Aut}(\ES{D})=\mathfrak{S}_{3}$. Supposons ces hypoth\`eses v\'erifi\'ees. Notons $L$ l'extension quadratique de $F$ contenue dans $E$ telle que l'image de $\Gamma_{L}$ dans $\mathfrak{S}_{3}$ soit le sous-groupe distingu\'e d'ordre $3$ de ce groupe. On fixe un \'el\'ement $\gamma\in \Gamma_{L}$ dont l'image dans $\Gamma_{E/L}$ engendre ce groupe. Appliquons (3) et fixons $\omega\in \Omega^{\sharp}$ tel que $\omega\gamma_{G'_{1}}\omega^{-1}=\gamma_{G'_{2}}$. Le m\^eme raisonnement que ci-dessus montre que les homomorphismes $\sigma\mapsto \sigma_{G'_{2}}$ et $\sigma\mapsto \omega\sigma_{G'_{1}}\omega^{-1}$ co\"{\i}ncident sur $\Gamma_{L}$. Soit $\sigma\in \Gamma_{F}$. On a $\sigma\gamma\sigma^{-1}\in \Gamma_{L}$, donc, en appliquant ce que l'on vient de d\'emontrer \`a cet \'el\'ement et \`a $\gamma$, on a les deux \'egalit\'es
$$
\sigma_{G'_{2}}\gamma_{G'_{2}}\sigma_{G'_{2}}^{-1}=\omega\sigma_{G'_{1}}\gamma_{G'_{1}}\sigma_{G'_{1}}^{-1}\omega^{-1},\quad\gamma_{G'_{2}}=\omega\gamma_{G'_{1}}\omega^{-1}.
$$
Posons
$$
\omega_{1}=\omega^{-1}\sigma_{G'_{2}}^{-1}\omega\sigma_{G'_{1}}=\omega^{-1}\sigma_{G}^{-1}w_{G'_{2}}(\sigma)^{-1}\omega w_{G'_{1}}(\sigma)\sigma_{G}.
$$
C'est un \'el\'ement de $\Omega$ puisque ce groupe est distingu\'e dans ${\rm Aut}(\ES{D}_{\rm a})$. Les \'egalit\'es pr\'ec\'edentes entra\^{\i}nent $\omega_{1}\gamma_{G'_{1}}\omega_{1}^{-1}=\gamma_{G'_{1}}$, c'est-\`a-dire $\omega_{1}w_{G'_{1}}(\gamma)\gamma_{G} =w_{G'_{1}}(\gamma)\gamma_{G}\omega_{1}$.  Parce que $\Omega$ est commutatif, on peut supprimer les termes $w_{G'_{1}}(\gamma)$: on a 
\begin{enumerate}[leftmargin=17pt] 
\item[(5)] $\omega_{1}\gamma_{G} =\gamma_{G}\omega_{1}$.
\end{enumerate} 
Notons $(\alpha_{i})_{i=1,\ldots,4}$ les \'el\'ements de $\Delta$, $\alpha_{2}$ \'etant la racine centrale, li\'ee aux trois autres. Les racines $\alpha\in \Delta_{\rm a}$ telles que $d(\alpha)=1$ sont les $\alpha_{i}$ pour $i=0,\,1,\,3,\,4$. Donc $\omega_{1}(\alpha_{0})$ est l'une de ces racines d'apr\`es \ref{hypothses et dŽfinitions}.(4).   Par hypoth\`ese, $\gamma_{G}$ est un automorphisme d'ordre $3$. Il fixe $\alpha_{2}$ et $\alpha_{0}$ et permute cycliquement $\alpha_{1},\alpha_{3},\alpha_{4}$. L'\'egalit\'e (5) entra\^{\i}ne que $\omega_{1}(\alpha_{0})$ est fixe par $\gamma_{G}$ et ne peut donc \^etre que $\alpha_{0}$. L'assertion \ref{hypothses et dŽfinitions}.(4) entra\^{\i}ne alors $\omega_{1}=1$. En revenant \`a la d\'efinition de $\omega_{1}$, cela signifie que $\sigma_{G'_{2}}=\omega\sigma_{G'_{1}}\omega^{-1}$.  Cela \'etant vrai  pour tout $\sigma\in \Gamma_{F}$, nos deux donn\'ees endoscopiques sont \'equivalentes. Cela ach\`eve la preuve.
 \end{proof}

\subsection{Equivalence presque partout, cas g\'en\'eral}\label{cas gŽnŽral}Dans ce paragraphe, on prouve la proposition de \ref{le rŽsultat} dans le cas gŽnŽral. 

Consid\'erons deux  donn\'ees endoscopiques $\bs{G}'_{1}=(G'_{1},\ES{G}'_{1},s_1)$ et $\bs{G}'_{2}=(G'_{2},\ES{G}'_{2},s_2)$ de $G$. 
On suppose que les donn\'ees $\bs{G}'_{1,v}$ et $\bs{G}'_{2,v}$ sont \'equivalentes pour presque toute place $v$.
On doit montrer que les donnŽes $\bs{G}'_{1}$ et $\bs{G}'_{2}$ sont \'equivalentes. 

Encore une fois, l'\'equivalence en au moins une place entra\^{\i}ne que $s_{1}$ et $s_{2}$ sont conjugu\'es. Quitte \`a remplacer l'une de nos donn\'ees par une donn\'ee \'equivalente, on peut supposer $s_{1}=s_{2}$ et on note simplement $s$ cet \'el\'ement. On va utiliser la construction de Langlands, cf. \cite[pp.~704-705]{L}. Notons $\mu_{\infty}$ le groupe des racines de l'unit\'e dans ${\mathbb C}^{\times}$. Le groupe ${\mathbb C}^{\times}/\mu_{\infty}$ est  sans torsion. Notons $p:{\mathbb C}^{\times}\to {\mathbb C}^{\times}/\mu_{\infty}$ la projection naturelle et  $R$ le sous-groupe engendr\'e par l'ensemble $\{p(\alpha(s));\alpha\in \Sigma\}$. C'est un groupe ab\'elien de type fini et sans torsion, donc il est isomorphe \`a ${\mathbb Z}^n$ pour un certain $n\in {\mathbb N}$. Fixons-en une base $(r_{i})_{i=1,\ldots,n}$. Un \'el\'ement $r\in R$ s'\'ecrit $r=\prod_{i=1,\ldots,n}r_{\! i}^{n_{i}(r)}$ avec $n_{i}(r)\in {\mathbb Z}$. Notons $R^+$, resp. $R^-$, le sous-ensemble des $r\in R$ tels que, ou bien $r=1$, ou bien, pour le plus petit entier $i$ tel que $n_{i}(r)\not=0$, on a $n_{i}(r)>0$, resp. $n_{i}(r)<0$. Il est clair que $R^+$ et $R^-$ sont stables par multiplication, que $R^-$ est l'ensemble des inverses d'\'el\'ements de $R^+$ et que $R^+\cap R^-=\{1\}$. Notons $\Sigma^P$, resp. $\Sigma^M$, l'ensemble des $\alpha\in \Sigma$ tels que $p(\alpha(s))\in R^+$, resp. $p(\alpha(s))=1$. Les propri\'et\'es pr\'ec\'edentes entra\^{\i}nent qu'il existe un sous-groupe parabolique $\hat{P}$ de $\hat{G}$ et une composante de Levi $\hat{M}$ de $\hat{P}$ tels que $\hat{T}\subset \hat{M}$ et que $\Sigma^P$, resp. $\Sigma^M$, soit l'ensemble des racines de $\hat{T}$ dans $\hat{P}$, resp. $\hat{M}$. Conjuguer nos donn\'ees par un m\^eme \'el\'ement de $N_{\hat{G}}(\hat{T})$ conjugue notre couple $(\hat{P},\hat{M})$ par le m\^eme \'el\'ement. On peut donc supposer $\hat{B}\subset \hat{P}$. Il r\'esulte de la construction que
\begin{enumerate}[leftmargin=17pt] 
\item[(1)]tout automorphisme de $\hat{G}$ qui conserve $\hat{T}$ et fixe $s$ conserve aussi la paire $(\hat{P},\hat{M})$.
\end{enumerate}

Notons $\Delta^M=\Delta\cap \Sigma^M$  et $\Delta_{M}=\Delta\smallsetminus \Delta^M$. L'ensemble $\Delta^M$ est un ensemble de racines simples pour $\hat{M}$. Introduisons la relation d'\'equivalence dans $\Delta_{M}$ d\'efinie par
$$
\alpha \equiv \beta \Leftrightarrow p(\alpha(s))=p(\beta(s)).
$$ Notons  $(\Delta_{i})_{i=1,\ldots,m}$ l'ensemble des classes d'\'equivalence. 
Pour $i=1,\ldots,m$, fixons une racine $\delta_{i}\in \Delta_{i}$. Fixons un entier $b\in {\mathbb N}_{>0}$ tel que:
 \begin{itemize}
 \item $\alpha(s)^b=1$ pour tout $\alpha\in \Sigma^M$;
 \item $\alpha(s)^b=\beta(s)^{b}$ pour tous $\alpha,\beta\in \Sigma$ tels que $p(\alpha(s))=p(\beta(s))$.
 \end{itemize}
Posons $c=\sum_{i=1,\ldots,m}\sum_{\alpha\in \Delta_{i}}id(\alpha)$ et $d=3cb$. 
D\'efinissons un \'el\'ement $t\in T$ par les \'egalit\'es:
\begin{itemize}
\item $\alpha(t)=\alpha(s)$ pour tout $\alpha\in \Delta^M$;
\item pour $i=1,\ldots,m$ et pour tout $\alpha\in \Delta_{i}$, $\alpha(t)=\zeta_{d}^{i}\alpha(s)\alpha(\delta_{i})^{-1}$.
\end{itemize}

Montrons que 
\begin{enumerate}[leftmargin=17pt] 
\item[(2)]tout automorphisme de $\hat{G}$ qui conserve $\hat{T}$ et fixe $t$ conserve aussi la paire $(\hat{P},\hat{M})$.
\end{enumerate}
Pour tout $\alpha\in \Sigma$, on a $\alpha(t)=\zeta_{d}^{m(\alpha)}$ pour un \'el\'ement bien d\'efini $m(\alpha)\in {\mathbb Z}/d{\mathbb Z}$. Introduisons les ensembles $S^M$,  resp. $S^{U}$, $S^{\bar{U}}$, des racines $\alpha\in \Sigma$ qui v\'erifient les conditions $m(\alpha)\in 3c{\mathbb Z}/d{\mathbb Z}$, resp. $m(\alpha)\in( \{1,\ldots ,c\}+3c{\mathbb Z} )/d{\mathbb Z}$, resp. $m(\alpha)\in( \{-1,\ldots ,-c\}+3c{\mathbb Z} )/d{\mathbb Z}$. Il r\'esulte des d\'efinitions que l'on a les inclusions
$\Sigma^M\subset S^M$, $(\Sigma^P\smallsetminus \Sigma^M)\subset S^{U}$ et $\{-\alpha; \alpha\in \Sigma^P-\Sigma^M\}\subset S^{\bar{U}}$. Il r\'esulte aussi des d\'efinitions que les ensembles $S^M$, $S^{U}$ et $S^{\bar{U}}$ sont disjoints. Les inclusions pr\'ec\'edentes sont des \'egalit\'es. Par d\'efinition de ces ensembles, $S^M$, $S^{U}$ et $S^{\bar{U}}$  sont conserv\'es par tout automorphisme de $\hat{G}$ qui conserve $\hat{T}$ et qui fixe $t$. Il en est donc de m\^eme de $\Sigma^P$ et $\Sigma^M$, donc de $\hat{P}$ et $\hat{M}$. Cela prouve (2). 

Montrons que
\begin{enumerate}[leftmargin=17pt] 
 \item[(3)]tout automorphisme de $\hat{G}$ qui conserve $\hat{T}$ fixe $s$ si et seulement s'il fixe $t$.
 \end{enumerate}
On consid\`ere un automorphisme $u$ de $\hat{G}$ qui conserve $\hat{T}$ et qui fixe l'un des \'el\'ements $s$ ou $t$. On veut prouver qu'il fixe l'autre. D'apr\`es (1) et (2), $u$ conserve la paire $(\hat{P},\hat{M})$. On sait que l'on peut \'ecrire $u= {\rm Int}_n \circ v$ pour un \'el\'ement $n\in N_{\hat{G}}(\hat{T})$ et un automorphisme $v$ de $\hat{G}$ qui conserve la paire $(\hat{B},\hat{T})$. Alors la paire $(v(\hat{P}),v(\hat{M}))$ est conjugu\'ee par $n^{-1}$ \`a $(\hat{P},\hat{M})$. Les deux paires $(\hat{P},\hat{M})$ et $(v(\hat{P}),v(\hat{M}))$ \'etant standard, il en r\'esulte qu'elles sont \'egales. Donc $v$ et ${\rm Int}_n$ conservent la paire $(\hat{P},\hat{M})$. Il en r\'esulte que $n\in N_{\hat{M}}(\hat{T})$. On note $w$ l'image de $n$ dans le groupe de Weyl $W^M$ de $\hat{M}$ relatif \`a $\hat{T}$.  Puisque $v$ conserve $(\hat{P},\hat{M})$,  $v$ conserve $\Delta^M$ et $\Delta_{M}$. 

Montrons que
\begin{enumerate}[leftmargin=17pt] 
\item[(4)] pour $i=1,\ldots ,m$, $v$ conserve $\Delta_{i}$.
\end{enumerate}
Notons $y$ l'\'el\'ement $s$ ou $t$ que l'on suppose conserv\'e par $u$. On a $wv(y)=y$. Soit $\alpha\in \Delta_{i}$, et posons $\beta=v(\alpha)$. On a $\beta(y)=v(\alpha)(wv(y))=w'(\alpha)(y)$, o\`u $w'=v^{-1}w^{-1}v$. Cet \'el\'ement $w'$ appartient encore \`a $W^M$. Il est connu que $w'(\alpha)$ s'\'ecrit $\alpha+\sum_{\alpha'\in \Delta^M}c(\alpha')\alpha'$, avec des coefficients $c(\alpha')\in {\mathbb Z}$. Par d\'efinition de $t$ et $y$, on a $\alpha'(y)=\alpha'(s)$ pour tout $\alpha'\in \Delta^M$ et ce terme est une racine de l'unit\'e d'ordre au plus $b$. Il en r\'esulte que $\beta(y)^b\alpha(y)^{-b}=1$. Si $y=s$, alors $p(\alpha(s))=p(\beta(s))$ donc $\beta\in \Delta_{i}$ par d\'efinition des classes d'\'equivalence $\Delta_{i}$. Si $y=t$, on a $m(\beta)\in (m(\alpha)+3c{\mathbb Z})/d{\mathbb Z}=(i+3c{\mathbb Z})/d{\mathbb Z}$ et l'on voit que cette relation caract\'erise les \'el\'ements de $\Delta_{i}$, donc encore $\beta\in \Delta_{i}$. Cela prouve (4).
 
Posons $x=s/t$. Pour achever la preuve de (3), il suffit de prouver que $w(x)=v(x)=x$. On a $\alpha(x)=1$ pour tout $\alpha\in \Delta^M$, ce qui implique que $x$ appartient au centre de $\hat{M}$ donc est fix\'e par $w\in W^M$. Pour $i=1,\ldots,m$, $\alpha(x)$ est constant quand $\alpha$ parcourt $\Delta_{i}$, cela par d\'efinition de $t$. Avec (4), on voit que $v(\alpha)(x)=\alpha(x)$ pour tout $\alpha\in \Delta$, donc $v(x)=x$. Cela ach\`eve de prouver (3).

Posons $\bs{G}''_{1}=(G'_{1},\ES{G}'_{1},t)$ et $\bs{G}''_{2}=(G'_{2},\ES{G}'_{2},t)$. Il r\'esulte facilement de (3) que ces triplets sont des donn\'ees endoscopiques de $G$. Il en r\'esulte tout aussi facilement que les donn\'ees $\bs{G}'_{1}$ et $\bs{G}'_{2}$ sont \'equivalentes si et seulement si les donn\'ees $\bs{G}''_{1}$ et $\bs{G}''_{2}$ le sont. La m\^eme chose vaut pour les donn\'ees localis\'ees en une place $v$. Alors l'assertion de la proposition pour $\bs{G}'_{1}$ et $\bs{G}'_{2}$ \'equivaut \`a la m\^eme assertion pour la paire $\bs{G}''_{1}$ et $\bs{G}''_{2}$. Mais on peut appliquer \`a cette derni\`ere paire la proposition de \ref{cas particulier}. Cela ach\`eve la preuve de la proposition de \ref{le rŽsultat} dans le cas o $\hat{G}=\hat{G}_{\rm AD}$ est simple (pour la forme \og galoisienne\fg des $L$-groupes). Cette preuve s'adapte aisŽment ˆ la forme \og groupes de Weil\fg des $L$-groupes, ce qui d'aprs \ref{rŽduction} 
prouve 
la proposition de \ref{le rŽsultat} dans le cas gŽnŽral.  

 \section{Description des donn\'ees endoscopiques elliptiques}
 
\subsection{Construction de donn\'ees endoscopiques elliptiques}\label{construction}
Les hypoth\`eses sont comme en \ref{hypothses et dŽfinitions}.
Notons $\underline{\ES{E}}(G)$ l'ensemble des couples $( \omega_{G'}, \ES{O})$ o\`u: 
\begin{itemize}
\item $\omega_{G'}:\Gamma_{F}\to \Omega$ est une application telle que l'application $\sigma\mapsto \sigma_{G'}\buildrel {\rm d\acute{e}f}\over{=}\omega_{G'}(\sigma)\sigma_{G}$ soit un homomorphisme de $\Gamma_{F}$ dans ${\rm Aut}(\ES{D}_{\rm a})$;
\item $\ES{O}$ est un sous-ensemble non vide de $\Delta_{\rm a}$ qui est conserv\'e par l'action $\sigma\mapsto \sigma_{G'}$ et qui forme une unique orbite pour cette action.
\end{itemize}
Disons que deux tels couples $(\omega_{G'_{1}},\ES{O}_{1})$ et $(\omega_{G'_{2}},\ES{O}_{2})$ sont \'equivalents si et seulement s'il existe $\omega\in \Omega$ tel que $\ES{O}_{2}=\omega(\ES{O}_{1})$ et $\sigma_{G'_{2}}=\omega\sigma_{G'_{1}}\omega^{-1}$ pour tout $\sigma\in \Gamma_{F}$. Notons $\underline{E}(G)$ l'ensemble des classes d'\'equivalence dans $\underline{\ES{E}}(G)$. 
Nous allons associer \`a tout \'el\'ement de $\underline{\ES{E}}(G)$ une donn\'ee endoscopique elliptique de $G$. 

ConsidŽrons un couple $(\omega_{G'},\ES{O})\in \underline{\ES{E}}(G)$. Posons
$$
d=\sum_{\alpha\in \ES{O}}d(\alpha).
$$
Soit $s$ l'unique \'el\'ement de $\hat{T}$ tel que $\alpha(s)=1$ pour tout $\alpha\in \Delta\smallsetminus (\ES{O}\cap \Delta)$ et $\alpha(s)=\zeta_{d}$ pour tout 
$\alpha\in \ES{O}\cap \Delta$. La relation \ref{hypothses et dŽfinitions}.(1) implique que ces \'egalit\'es s'\'etendent \`a $\Delta_{\rm a}$, c'est-\`a-dire $\alpha(s)=1$ pour tout 
$\alpha\in \Delta_{\rm a}\smallsetminus \ES{O}$ et $\alpha(s)=\zeta_{d}$ pour tout $\alpha\in \ES{O}$. Puisque $\ES{O}$ est stable par l'action galoisienne $\sigma\mapsto \sigma_{G'}$, le point $s$ est fixe par cette action. 
Posons $\hat{G'}=Z_{\hat{G}}(s)^\circ$. Ce groupe contient $\hat{T}$. Notons $\ES{G}'$ le sous-ensemble de $^LG$ form\'e des \'el\'ements $(g\tilde{\omega}_{G'}(\sigma),\sigma)$ pour $g\in \hat{G'}$ et $\sigma\in \Gamma_{F}$, o\`u $\tilde{\omega}_{G'}(\sigma)$ est un repr\'esentant quelconque de $\omega_{G'}(\sigma)$ dans $N_{\hat{G}}(\hat{T})$. L'ensemble $\ES{G}'$ un  groupe qui normalise $\hat{G}'$. On en d\'eduit de la fa\c{c}on habituelle une $L$-action de $\Gamma_{F}$ sur $\hat{G}'$ et on introduit le groupe r\'eductif connexe $G'$ d\'efini et quasi-d\'eploy\'e sur $F$ dont $\hat{G}'$, muni de cette action galoisienne, est le groupe dual. On v\'erifie que le triplet $\bs{G}'=(G'\!,\ES{G}'\!,s)$ est une donn\'ee endoscopique de $G$.  

Puisque $\ES{O}$ est non vide, les \'el\'ements de $\Delta_{\rm a}\smallsetminus \ES{O}$ sont lin\'eairement ind\'ependants. Notons $\Sigma_{0}$ l'ensemble des \'el\'ements de $\Sigma$ qui sont combinaisons lin\'eaires des \'el\'ements de $\Delta_{\rm a}\smallsetminus \ES{O}$ \`a coefficients entiers relatifs. Notons $\Sigma_{0}^+$, resp. $\Sigma_{0}^-$, le sous-ensemble des \'el\'ements de $\Sigma_{0}$ pour lesquels ces coefficients sont tous positifs ou nuls, resp. n\'egatifs ou nuls. \'Evidemment, $\Sigma_{0}^-=-\Sigma_{0}^+$. Notons $\Sigma^{G'}$ l'ensemble des racines de $\hat{T}$ dans $\hat{G}'$. 
Montrons:
\begin{enumerate}[leftmargin=17pt] 
\item[(1)]on a les \'egalit\'es $\Sigma^{G'}=\Sigma_{0}=\Sigma_{0}^+\cup \Sigma_{0}^-$.
\end{enumerate}
L'ensemble $\Sigma^{G'}$ est celui des $\alpha\in \Sigma$ tels que $\alpha(s)=1$. L'inclusion $\Sigma_{0}\subset \Sigma^{G'}$ est \'evidente. Soit $\beta\in \Sigma^{G'}$. Quitte \`a remplacer $\beta$ par $-\beta$, on suppose $\beta\in \Sigma^+$. \'Ecrivons $\beta=\sum_{\alpha\in \Delta}m(\alpha)\alpha$, avec des coefficients $m(\alpha)\in {\mathbb N}$. D'apr\`es \ref{hypothses et dŽfinitions}.(3),  on a $m(\alpha)\leq d(\alpha)$ pour tout $\alpha\in \Delta$. Posons $m=\sum_{\alpha\in \ES{O}\cap\Delta}m(\alpha)$. Les in\'egalit\'es pr\'ec\'edentes impliquent $0\leq m\leq d$ et m\^eme $m\leq d-1$ si $\alpha_{0}\in \ES{O}$.  L'\'egalit\'e $\beta(s)=1$ implique $\zeta_{d}^m=1$, donc $m$ est divisible par $d$. Si $\alpha_{0}\in \ES{O}$ cela force $m=0$. Si $\alpha_{0}\not\in\ES{O}$, on a seulement $m=0$ ou $m=d$. Si $m=0$, on a $m(\alpha)=0$ pour tout $\alpha\in \ES{O}\cap \Delta$, donc $\beta=\sum_{\alpha\in \Delta\smallsetminus (\ES{O}\cap \Delta)}m(\alpha)\alpha$ et $\beta$ appartient \`a $\Sigma_{0}^+$. Supposons $m=d$, ce qui impose $\alpha_{0}\not\in \ES{O}$. L'\'egalit\'e $m=d$ et les in\'egalit\'es $m(\alpha)\leq d(\alpha) $ pour tout $\alpha\in \ES{O}\cap \Delta=\ES{O}$ impliquent $m(\alpha)=d(\alpha)$ pour tout $\alpha\in \ES{O}$. Alors
\begin{eqnarray*}
\beta &=& \left(\sum_{\alpha\in \Delta}(m(\alpha)-d(\alpha))\alpha\right)+\left(\sum_{\alpha\in \ES{O}}d(\alpha)\alpha\right)\\
&=&\left(\sum_{\alpha\in \Delta\smallsetminus \ES{O}}(m(\alpha)-d(\alpha))\alpha\right)-\alpha_{0}.
\end{eqnarray*}
C'est une combinaison lin\'eaire \`a coefficients entiers n\'egatifs ou nuls d'\'el\'ements de $\Delta_{\rm a}\smallsetminus \ES{O}$. Donc $\beta\in \Sigma_{0}^-$. Cela prouve (1). 

Cette propri\'et\'e entra\^{\i}ne qu'il existe un unique sous-groupe de Borel $\hat{B}'$ contenant $\hat{T}$ de sorte que l'ensemble de racines positives dans $\hat{G}'$ associ\'e \`a ce Borel soit $\Sigma_{0}^+$. L'ensemble de racines simples $\Delta^{G'}$ est alors $\Delta_{\rm a}\smallsetminus \ES{O}$. 

Montrons que
\begin{enumerate}[leftmargin=17pt] 
\item[(2)] la donn\'ee $\bs{G}'$ est elliptique.
\end{enumerate}
Notons $X^*(\hat{T})$ le groupe des caract\`eres alg\'ebriques de $\hat{T}$ et $X^*(\hat{T})_{{\mathbb Q}}=X^*(\hat{T})\otimes_{{\mathbb Z}}{\mathbb Q}$. Pour tout ensemble fini $E\subset \Delta_{\rm a}$ stable par l'action galoisienne $\sigma\mapsto \sigma_{G'}$, notons ${\mathbb Q}[E]$ l'espace vectoriel sur ${\mathbb Q}$ de base $E$. Si $E\not=\Delta_{\rm a}$ l'espace ${\mathbb Q}[E]$ se plonge naturellement dans $X^*(\hat{T})_{{\mathbb Q}}$. On munit $X^*(\hat{T})_{{\mathbb Q}}$ et ${\mathbb Q}[E]$ de  cette action galoisienne  $\sigma\mapsto \sigma_{G'}$. Dire que la donnŽe $\bs{G}'$ est elliptique est \'equivalent \`a l'\'egalit\'e
$$
X^*(\hat{T})_{{\mathbb Q}}^{\Gamma_{F}}={\mathbb Q}[\Delta^{G'}]^{\Gamma_{F}},
$$
les exposants signifiants que l'on prend les sous-espaces de points fixes. On a une projection naturelle $p:{\mathbb Q}[\Delta_{\rm a}]\to X^*(\hat{T})_{{\mathbb Q}}$ dont le noyau est l'espace engendr\'e par la relation (1) de \ref{hypothses et dŽfinitions}, laquelle est fixe par $\Gamma_{F}$. De plus, cette projection est l'identit\'e sur ${\mathbb Q}[\Delta^{G'}]$ 
(qui est plong\'e \`a la fois dans les espaces de d\'epart et d'arriv\'ee). L'ellipticit\'e se traduit donc par l'\'egalit\'e
$$
\dim({\mathbb Q}[\Delta_{\rm a}]^{\Gamma_{F}})=1+\dim({\mathbb Q}[\Delta^{G'}]^{\Gamma_{F}}).
$$
Cela \'equivaut \`a $\dim({\mathbb Q}[\ES{O}]^{\Gamma_{F}})=1$. Mais ceci r\'esulte de l'hypoth\`ese que $\ES{O}$ est une unique orbite pour l'action de $\Gamma_{F}$.

\subsection{InjectivitŽ}\label{injectivitŽ}
Continuons avec les notations de \ref{construction}. 
Munissons $\Sigma$ de l'ordre partiel $\leq_{\hat{B}'}$ d\'efini en \ref{la construction de Langlands}: pour $\alpha,\beta\in \Sigma$, on a $\alpha\leq_{\hat{B}'} \beta$ si et seulement si $\beta-\alpha$ est combinaison lin\'eaire \`a coefficients dans ${\mathbb N}$ d'\'el\'ements de $\Delta_{\rm a}\smallsetminus \ES{O}$. Posons
$$
\mathfrak{S}=\{\alpha\in \Sigma: \alpha(s)=\zeta_{d}\}.
$$
Cet ensemble contient $\ES{O}$. 

Montrons que
\begin{enumerate}[leftmargin=17pt] 
\item[(1)] pour tout $\beta\in \mathfrak{S}$, il existe $\alpha\in \ES{O}$ tel que $\alpha\leq_{\hat{B}'}\beta$.
\end{enumerate}
Supposons d'abord $\beta\in \Sigma^+$. \'Ecrivons $\beta=\sum_{\alpha\in \Delta}m(\alpha)\alpha$. Posons $m=\sum_{\alpha\in \ES{O}\cap \Delta}m(\alpha)$. Comme plus haut, on a $0\leq m(\alpha)\leq d(\alpha)$ pour tout $\alpha\in \Delta$, d'o\`u $0\leq m\leq d$. On a $\beta(s)=\zeta_{d}^m$. Puisque $\beta(s)=\zeta_{d}$, cela implique $m\equiv 1\;({\rm mod}\,d{\mathbb Z})$, donc $m=1$. Il existe donc un unique \'el\'ement $\alpha\in \ES{O}\cap \Delta$ tel que $m(\alpha)\not=0$ et, pour cette unique racine, on a $m(\alpha)=1$. Alors $\beta=\alpha+\sum_{\alpha'\in \Delta\smallsetminus (\ES{O}\cap \Delta)}m(\alpha')\alpha'$, donc $\alpha\leq_{\hat{B}'}\beta$. Supposons maintenant $\beta\in -\Sigma^+$, \'ecrivons  $-\beta=\sum_{\alpha\in \Delta}m(\alpha)\alpha$ et d\'efinissons $m$ comme ci-dessus. On a alors $\beta(s)=\zeta_{d}^{-m}$ d'o\`u $-m\equiv 1\;({\rm mod}\,d{\mathbb Z})$. Cela impose $m=d-1$. Supposons d'abord $\alpha_{0}\not\in \ES{O}$. On a alors $m=\sum_{\alpha\in \ES{O}}m(\alpha)$ et $d=\sum_{\alpha\in \ES{O}}d(\alpha)$. Les in\'egalit\'es entre les $m(\alpha)$ et les $d(\alpha)$ impliquent qu'il existe un unique \'el\'ement $\alpha\in \ES{O}$ tel que $m(\alpha)=d(\alpha)-1$ et on a $m(\alpha')=d(\alpha')$ pour tout $\alpha'\in \ES{O}\smallsetminus \{\alpha\}$. Alors
\begin{eqnarray*}
\beta &=& \alpha_{0}+\sum_{\alpha'\in \Delta}(d(\alpha')-m(\alpha'))\alpha'\\
&=&\alpha+\alpha_{0}+\sum_{\alpha'\in \Delta\smallsetminus \ES{O}}(d(\alpha')-m(\alpha'))\alpha'.
\end{eqnarray*}
Cette \'egalit\'e montre que $\alpha\leq_{\hat{B}'}\beta$. Supposons maintenant $\alpha_{0}\in \ES{O}$. On a alors $m=\sum_{\alpha\in \ES{O}\smallsetminus \{\alpha_{0}\}}m(\alpha)$ et $d=1+\sum_{\alpha\in \ES{O}\smallsetminus \{\alpha_{0}\}}d(\alpha)$. L'\'egalit\'e $m=d-1$ force $m(\alpha)=d(\alpha)$ pour tout $\alpha\in \ES{O}\cap \Delta$. On obtient comme ci-dessus
$$\beta=\alpha_{0}+\sum_{\alpha'\in \Delta\smallsetminus (\ES{O}\cap \Delta)}(d(\alpha')-m(\alpha'))\alpha'$$
et $\alpha_{0}\leq_{\hat{B}'}\beta$. Cela d\'emontre (1).

Montrons que
\begin{enumerate}[leftmargin=17pt] 
\item[(2)] $\ES{O}$ est l'ensemble des \'el\'ements de $\mathfrak{S}$ qui sont minimaux pour l'ordre $\leq_{\hat{B}'}$.
\end{enumerate}
La relation (1) entra\^{\i}ne que l'ensemble des \'el\'ements de $\mathfrak{S}$ minimaux pour l'ordre $\leq_{\hat{B}'}$ appartiennent \`a $\ES{O}$. Soit maintenant $\alpha\in \ES{O}$ et consid\'erons un \'el\'ement $\beta\in \mathfrak{S}$ tel que $\beta\leq_{\hat{B}'}\alpha$. D'apr\`es (1), on peut fixer $\alpha'\in \ES{O}$ tel que $\alpha'\leq_{\hat{B}'}\beta$, donc $\alpha'\leq_{\hat{B}'}\alpha$. Supposons $\alpha'\not=\alpha$. Alors, par d\'efinition de l'ordre $\leq_{\hat{B}'}$, on obtient une \'egalit\'e
$$
-\alpha+\alpha'+\sum_{\alpha''\in \Delta_{\rm a}\smallsetminus \ES{O}}m(\alpha'')\alpha''=0,
$$
avec des coefficients $m(\alpha'')\in {\mathbb N}$. Or les seules relations lin\'eaires entre \'el\'ements de $\Delta_{\rm a}$ sont proportionnelles \`a  la relation (1) de \ref{hypothses et dŽfinitions} donc le coefficient de $\alpha$ et celui de $\alpha'$ doivent \^etre de m\^eme signe. Ce n'est pas le cas. Cette contradiction prouve que $\alpha'=\alpha$. Alors, les in\'egalit\'es $\alpha=\alpha'\leq_{\hat{B}'}\beta\leq_{\hat{B}'}\alpha$ entra\^{\i}nent $\beta=\alpha$, ce qui prouve que $\alpha$ est minimal. Cela d\'emontre (2).

\begin{mapropo}{L'application qui, \`a un couple $(\omega_{G'},\ES{O})\in \underline{\ES{E}}(G)$, associe la donn\'ee $\bs{G}'\!$, se quotiente en une injection de $\underline{E}(G)$ dans l'ensemble des classes d'\'equivalence de donn\'ees endoscopiques elliptiques de $G$.}
\end{mapropo}

\begin{proof}Soient  $(\omega_{G'_{1}},\ES{O}_{1}),(\omega_{G'_{2}},\ES{O}_{2})\in \underline{\ES{E}}(G)$. Notons $\bs{G}'_{1}$ et $\bs{G}'_{2}$ les donn\'ees endoscopiques de $G$ associ\'ees. La proposition signifie que $(\omega_{G'_{1}},\ES{O}_{1})$ et $(\omega_{G'_{2}},\ES{O}_{2})$ sont \'equivalentes si et seulement si $\bs{G}'_{1} $ et $\bs{G}'_{2}$ le sont. Dans un sens, c'est clair: si les couples $(\omega_{G'_{1}},\ES{O}_{1})$ et $(\omega_{G'_{2}},\ES{O}_{2})$ sont conjugu\'es par un \'el\'ement $\omega\in \Omega$, il suffit de relever $\omega$ en un \'el\'ement quelconque $x$  de $N_{\hat{G}}(\hat{T})$. L'\'el\'ement $x$ est une \'equivalence entre les deux donn\'ees endoscopiques. Inversement, supposons que les donnŽes $\bs{G}'_{1}
$ et $\bs{G}'_{2}$ soient \'equivalentes. On affecte les objets relatifs \`a nos deux s\'eries de donn\'ees d'indices $1$ ou $2$. En particulier, on introduit les sous-groupes de Borel $\hat{B}'_{1}$ de $\hat{G}'_{1}$, resp. $\hat{B}'_{2}$ de $\hat{G}'_{2}$, que l'on a d\'efini avant l'assertion \ref{construction}.(2). Fixons $x\in \hat{G}$ qui r\'ealise l'\'equivalence, c'est-\`a-dire que $s_{2}=xs_{1}x^{-1}$ et $\ES{G}'_{2}=x\ES{G}'_{1}x^{-1}$. On peut multiplier $x$ \`a droite par un \'el\'ement de $\hat{G}'_{1}$ et \`a gauche par un \'el\'ement de $\hat{G}'_{2}$. On peut donc supposer que la conjugaison par $x$ envoie la paire de Borel $(\hat{B}'_{1},\hat{T})$ de $\hat{G}'_{1}$ sur la paire de Borel $(\hat{B}'_{2},\hat{T})$ de $\hat{G}'_{2}$. En particulier, $x$ normalise $\hat{T}$ et d\'efinit un \'el\'ement $\omega\in W$. Puisque, pour $i=1,\,2$, $\Delta_{\rm a}\smallsetminus \ES{O}_{i}$ est l'ensemble des racines simples dans $\Sigma^{G'_{i}}$ pour le Borel $\hat{B}'_{i}$, on a $\omega(\Delta_{\rm a}\smallsetminus \ES{O}_{1})=\Delta_{\rm a}\smallsetminus \ES{O}_{2}$. De m\^eme, $\omega$ transporte l'ordre $\leq_{\hat{B}'_1}$ en l'ordre $\leq_{\hat{B}'_2}$. Pour $i=1,2$, $d_{i}$ est l'ordre de $s_{i}$. Puisque $xs_{1}x^{-1}=s_{2}$, on a $d_{1}=d_{2}$ et on note simplement $d$ cet entier. Alors la conjugaison par $\omega$ envoie $\mathfrak{S}_{1}$ sur $\mathfrak{S}_{2}$. Puisque cette conjugaison transporte les ordres, l'assertion (2) entra\^{\i}ne qu'elle envoie $\ES{O}_{1}$ sur $\ES{O}_{2}$. Alors, elle conserve $\Delta_{\rm a}$, donc $\omega\in \Omega$. 
Soit $\sigma\in \Gamma_{F}$. Pour $i=1,\,2$,  l'\'el\'ement $ \omega_{G'_i}(\sigma)$ est obtenu en choisissant un ŽlŽment $(g'_i(\sigma),\sigma)\in \ES{G}'_i$ qui conserve 
$(\hat{B}'_i,\hat{T})$, puis en posant $\omega_{G'_i}(\sigma)=\hat{T}g'_i(\sigma)$. Puisque $x\ES{G}'_1 x^{-1}= \ES{G}'_2$, on a $x(g'_{1}(\sigma),\sigma)x^{-1}=(hg_{2}(\sigma),\sigma)$, pour un $h\in \hat{G}'$. Puisque la conjugaison par $x$ envoie $(\hat{B}'_1,\hat{T})$ sur $(\hat{B}'_2,\hat{T})$, la conjugaison par $h$ conserve $(\hat{B}'_2,\hat{T})$. Donc $h\in \hat{T}$. On en d\'eduit  $xg'_1(\sigma) \sigma_G(x)^{-1} \in \hat{T}g'_2(\sigma)$, ce qui Žquivaut ˆ $\omega \omega_{G'_1}(\sigma)\sigma_G(\omega)^{-1}= \omega_{G'_2}(\sigma)$.
 Ceci \'etant vrai pour tout $\sigma$, les couples $(\omega_{G'_{1}},\ES{O}_{1})$ et $(\omega_{G'_{2}},\ES{O}_{2})$ sont \'equivalents. \end{proof}

\subsection{Surjectivit\'e d'apr\`es  Langlands}\label{surjectivitŽ}

\begin{mapropo}
L'application qui, \`a un couple $(\omega_{G'},\ES{O})\in \underline{\ES{E}}(G)$, associe la donn\'ee $\bs{G}'\!$, se quotiente en une bijection de $\underline{E}(G)$ sur l'ensemble des classes d'\'equivalence de donn\'ees endoscopiques elliptiques de $G$.
\end{mapropo}

\begin{proof}
L'injectivitŽ ayant dŽjˆ ŽtŽ Žtablie, il faut prouver la surjectivit\'e. On consid\`ere une donn\'ee endoscopique elliptique $\bs{G}'=(G'\!,\ES{G}'\!,s)$ de $G$. Parce qu'elle est elliptique, Langlands prouve (par la m\'ethode que l'on a reprise en \ref{cas gŽnŽral}) que $s$ est d'ordre fini. On effectue la construction de \ref{la construction de Langlands}, dont en reprend les notations. Montrons que
\begin{enumerate}[leftmargin=17pt] 
\item[(1)]si la donnŽe $\bs{G}'$ est \'equivalente \`a la donn\'ee principale $\bs{G}=(G,1,{^LG})$, on a $\mathfrak{X}=\Delta$; sinon, on a $\mathfrak{X}=\Delta_{\rm a}$, il existe un unique $k\in \{1,\ldots ,d-1\}$ tel que $\mathfrak{X}_{k}\not=\emptyset$ et cet unique ensemble $\mathfrak{X}_{k}$ forme  une unique orbite pour l'action $\sigma\mapsto \sigma_{G'}$. 
\end{enumerate}
Si la donnŽe $\bs{G}'$ est \'equivalente \`a $\bs{G}$, on a $s=1$ donc $d=1$ et $\mathfrak{X}=\mathfrak{X}_{0}$. Puisque ce dernier ensemble est un ensemble de racines simples pour $\hat{G}'$, ses \'el\'ements sont lin\'eairement ind\'ependants donc $\mathfrak{X}\not=\Delta_{\rm a}$, ce qui implique $\mathfrak{X}=\Delta$. Supposons maintenant que la donnŽe $\bs{G}'$ ne soit pas \'equivalente \`a $\bs{G}$. Alors $s\not=1$ et $d\geq2$. Il y a forc\'ement une racine $\alpha\in \Sigma$ pour laquelle $\alpha(s)\not=1$, donc un $k\in\{1,\ldots ,d-1\}$ tel que $\mathfrak{Y}_{k}\not=\emptyset$. Consid\'erons le plus petit $k$ pour lequel cette propri\'et\'e est v\'erifi\'ee. Alors, par d\'efinition, $\mathfrak{Z}_{k}=\mathfrak{Y}_{k}$ donc $\mathfrak{Z}_{k}\not=\emptyset$. L'ensemble $\mathfrak{X}_{k}$ des \'el\'ements minimaux de $\mathfrak{Z}_{k}$ n'est pas vide non plus.  Comme dans la preuve de \ref{construction}.(2), l'ellipticit\'e se traduit par l'\'egalit\'e $\dim(X^*(\hat{T})_{{\mathbb Q}}^{\Gamma_{F}})=\dim({\mathbb Q}[\mathfrak{X}_{0}]^{\Gamma_{F}})$. Si $\mathfrak{X}=\Delta$, on a
$$
\dim(X^*(\hat{T})_{{\mathbb Q}}^{\Gamma_{F}} )=\dim({\mathbb Q}[\mathfrak{X}_{0}]^{\Gamma_{F}})+\sum_{j=1,\ldots ,d-1}\dim({\mathbb Q}[\mathfrak{X}_{j}]^{\Gamma_{F}}).
$$
Mais puisque $\mathfrak{X}_{k}$ n'est pas vide, l'espace $ {\mathbb Q}[\mathfrak{X}_{k}]^{\Gamma_{F}}$ n'est pas nul: la somme des \'el\'ements de $\mathfrak{X}_{k}$ appartient \`a cet espace. Donc $\dim(X^*(\hat{T})_{{\mathbb Q}}^{\Gamma_{F}})>\dim({\mathbb Q}[\mathfrak{X}_{0}]^{\Gamma_{F}})$, contradiction. Donc $\mathfrak{X}=\Delta_{\rm a}$. Comme dans la preuve de \ref{construction}(2), l'ellipticit\'e se traduit alors par l'\'egalit\'e
$$
\dim({\mathbb Q}[\Delta_{\rm a}]^{\Gamma_{F}})=1+\dim({\mathbb Q}[\mathfrak{X}_{0}]^{\Gamma_{F}}).
$$ S'il y a au moins deux $j\in \{1,\ldots ,d-1\}$ tels que $\mathfrak{X}_{j}\not=\emptyset$, on voit comme ci-dessus que  $\dim({\mathbb Q}[\Delta_{\rm a}]^{\Gamma_{F}})\geq 2+\dim({\mathbb Q}[\mathfrak{X}_{0}]^{\Gamma_{F}})$, contradiction. Donc $k$ est l'unique \'el\'ement de $\{1,\ldots ,d-1\}$ tel que $\mathfrak{X}_{k}\not=\emptyset$. Pour la m\^eme raison, on a $dim({\mathbb Q}[\mathfrak{X}_{k}]^{\Gamma_{F}})=1$. Cela prouve (1).

Consid\'erons le couple $(\omega_{G},\{\alpha_{0}\})$, o\`u $\omega_{G}:\Gamma_{F}\to \Omega$ est l'application constante de valeur $1$. Ce couple appartient \`a $\underline{\ES{E}}(G)$ et on voit que la donn\'ee endoscopique qui lui est associ\'ee n'est autre que $\bs{G}$. Cela r\`egle la question pour cette donn\'ee. Supposons maintenant que la donn\'ee $\bs{G}'$ ne soit pas \'equivalente \`a $\bs{G}$. Alors $\mathfrak{X}=\Delta_{\rm a}$ et la construction de \ref{la construction de Langlands} munit $\Delta_{\rm a}$ d'une action galoisienne $\sigma\mapsto \sigma_{G'}$ qui est de la forme $\sigma_{G'}=w_{G'}(\sigma)\sigma_{G}$, o\`u $w_{G'}$ est une application de $\Gamma_{F}$ dans $\Omega$. On pose $\ES{O}=\mathfrak{X}_{k}$, o\`u $k$ est l'entier de la relation (1). Le couple $(w_{G'},\ES{O})$ appartient \`a $\underline{\ES{E}}(G)$. On lui associe une donn\'ee endoscopique $\bs{G}''=(G'',\ES{G}'',\underline{s})$ comme dans le paragraphe pr\'ec\'edent. Il r\'esulte des d\'efinitions que $\ES{G}''=\ES{G}'$. Il est moins clair que $\underline{s}=s$. En effet, les d\'efinitions entra\^{\i}nent $\alpha(s)=\alpha(\underline{s})=1$ pour $\alpha\in \Delta_{\rm a}\smallsetminus \ES{O}$, mais, pour $\alpha\in \ES{O}$, $\alpha(s)=\zeta_{d}^k$ tandis que $\alpha(\underline{s})=\zeta_{\underline{d}}$, o\`u $\underline{d}=\sum_{\beta\in \ES{O}}d(\beta)$. Pour montrer que $\underline{s}=s$ et donc pour achever la preuve de la proposition, il reste \`a prouver:
\begin{enumerate}[leftmargin=17pt] 
\item[(2)] on a les \'egalit\'es $\underline{d}=d$ et $k=1$. 
\end{enumerate}
On vient de dire que $\alpha(s)=1$ si $\alpha\in \Delta_{\rm a}\smallsetminus \ES{O}$ et $\alpha(s)=\zeta_{d}^k$ si $\alpha\in \ES{O}$. Puisque $d$ est par d\'efinition l'ordre de $s$, cela entra\^{\i}ne   que $k$ est premier \`a $d$. De plus, la relation \ref{hypothses et dŽfinitions}.(1) entra\^{\i}ne que $\zeta_{d}^{k \underline{d}}=1$, donc $d$ divise $ \underline{d}$. 
On sait que, pour toute racine positive $\beta\in \Sigma$, on peut trouver une suite de racines positives $(\beta_{i})_{i=1,\ldots,m}$ de sorte que $\beta_{1}\in \hat{\Delta}$, $\beta_{i+1}-\beta_{i}\in\Delta$ pour $i=1,\ldots ,m-1$ et $\beta_{m}=\beta$. On applique cela \`a $\beta=-\alpha_{0}  =\sum_{\alpha\in \Delta}d(\alpha)\alpha$ et on \'ecrit 
$$
 \beta_{i}= \sum_{\alpha\in \Delta}d_{i}(\alpha)\alpha.\leqno{(3)}
$$
L'application $i\mapsto d_{i}=\sum_{\alpha\in \ES{O}\cap \Delta}d_{i}(\alpha)$ est croissante et  on a $d_{i+1}\leq d_{i}+1$. Pour $i=m$, on a $d_{m}=\underline{d}$ si $\alpha_{0}\not\in \ES{O}$, et $d_{m}=\underline{d}-1$ si $\alpha_{0}\in \ES{O}$.  Supposons $d<\underline{d}$. Alors il existe $i$ tel que $d_{i}=d$. Fixons un tel  $i$. On a alors $\beta_{i}(s)=\zeta_{d}^{kd}=1$. Donc $\beta_{i}\in \Sigma^{G'}$ et $\beta_{i}$ est combinaison lin\'eaire des \'el\'ements de $\mathfrak{X}_{0}=\Delta_{\rm a}\smallsetminus \ES{O}$. \'Ecrivons $\beta_{i}=\sum_{\alpha\in \Delta_{\rm a}\smallsetminus \ES{O}}m(\alpha)\alpha$. Si $\alpha_{0}\in \ES{O}$, seuls interviennent ici des \'el\'ements de $\Delta$ et cette \'egalit\'e co\"{\i}ncide avec (3). En ce cas $d_{i}(\alpha)=0$ pour $\alpha\in \ES{O}\cap \Delta$, donc $d=d_{i}=0$ ce qui est impossible. Si $\alpha_{0}\not\in \ES{O}$, on a
\begin{eqnarray*}
 \beta_{i} &= &m(\alpha_{0})\alpha_{0}+\sum_{\alpha\in \Delta\smallsetminus \ES{O}}m(\alpha)\alpha\\
& =& -\sum_{\alpha\in \ES{O}}m(\alpha_{0})d(\alpha)\alpha+\sum_{\alpha\in \Delta\smallsetminus \ES{O}}(m(\alpha)-m(\alpha_{0})d(\alpha))\alpha.
\end{eqnarray*}
 En comparant avec (3), on obtient
 $$
 d=-m(\alpha_{0})\sum_{\alpha\in \ES{O}}d(\alpha)=-m(\alpha_{0})\underline{d}.
 $$ On a $d\geq1$ par d\'efinition de $d$. L'\'egalit\'e pr\'ec\'edente contredit l'hypoth\`ese $ d<\underline{d}$. Cette hypoth\`ese est donc contradictoire, d'o\`u $d=\underline{d}$. 
  Le m\^eme raisonnement prouve que, pour tout $e\in \{1,\ldots ,d-1\}$, il existe une racine $\beta_{i}$ telle que $\beta_{i}(s)=\zeta_{d}^{ke}$.  En choisissant pour $e$ l'entier tel que 
  $ke\equiv 1\;({\rm mod}\,d{\mathbb Z})$, on obtient que l'ensemble $\mathfrak{Y}_{1}$ n'est pas vide. Comme on l'a vu ci-dessus, cela entra\^{\i}ne $\mathfrak{X}_{1}\not=\emptyset$, donc $k=1$.   Cela d\'emontre (2) et la proposition.
\end{proof}


\end{document}